\documentclass[11pt,a4paper]{article}
\usepackage{mycommands}
\usepackage{oubraces}
\usepackage{pdfpages}

\HIDE{
\usepackage[notref,notcite]{showkeys}
\usepackage{rotating}
\makeatletter
\newcommand*{\shiftright}[2]{%
  \settowidth{\@tempdima}{#2}%
  \makebox[\@tempdima]{\hspace*{#1}#2}%
}
\makeatother
\newlength{\keyheight}

} 

\title{\textbf{Intervals of permutation class growth rates 
}}

\author{$\phantom{{}^\dagger}$David Bevan${}^\dagger$}

\date{}

\begin{document}
\maketitle

{\let\thefootnote\relax\footnotetext
{${}^\dagger$Department of Mathematics and Statistics, The Open University, Milton Keynes, England.}}

{\let\thefootnote\relax\footnotetext
{2010 Mathematics Subject Classification:
05A05, 
05A16. 
}}

\begin{abstract}
\noindent
We prove that the set of growth rates of permutation classes
includes an infinite sequence of intervals whose
infimum
is $
\theta_B
\approx2.35526$,
and that it also
contains every value at least $\lambda_B\approx2.35698$.
These results improve on a theorem of Vatter, who determined that there are permutation classes of
every growth rate at least $\lambda_A\approx2.48187$. Thus, we also refute his conjecture that the set of growth rates
below $\lambda_A$ is nowhere dense.
Our proof
is based upon an analysis of expansions of real numbers in non-integer bases, 
the study of which was initiated by R\'enyi in the 1950s.
In particular,
we prove two generalisations of a result of Pedicini concerning %
expansions in which the digits are drawn from sets of allowed values. 
\end{abstract}

\section{Introduction}\label{sectLambdaPaperIntro}

We consider a permutation of length $n$ simply to be a sequence containing exactly one copy of each element of the set $\{1,\ldots,n\}$.
A permutation $\tau$ of length $k$ is said to be a \emph{subpermutation} of another permutation $\sigma$,
written $\tau\leqslant\sigma$, if $\sigma$ has a subsequence of length $k$ whose terms occur in the same relative order as those of $\tau$.
For example, $\mathbf{1324}$ is a subpermutation of $\mathbf{241635}$ since $\mathbf{2435}$ has the same relative order as $\mathbf{1324}$.

A \emph{permutation class} is a downset of permutations under this subpermutation 
order; thus if $\sigma$ is a member of a permutation class $\CCC$ and $\tau\leqslant\sigma$, then $\tau\in\CCC$.
If $\SSS$ is a set of permutations, we use $\SSS_n$ to denote the members of $\SSS$ of length $n$.
The \emph{growth rate} of a permutation class $\CCC$ is defined by the limit
$$
\gr(\CCC) \;=\;
\liminfty\sqrt[n]{|\CCC_n|},
$$
if it exists.
There are no known examples of permutation classes for which this limit does not exist 
and it is widely believed (\cite{Vatter2014} Conjecture~1.1) that every permutation class does indeed have a growth rate.
The proof by Marcus \& Tardos~\cite{MT2004} of the Stanley--Wilf conjecture implies that growth rates are finite except in the case of the class of all permutations.

What does the set of possible permutation class growth rates look like?
The last few years have seen substantial progress on answering this question, with particular focus on 
significant phase transition values.
Kaiser \& Klazar~\cite{KK2003} characterised all growth rates up to 2, showing that
the golden ratio $\varphi\approx1.61803$ (the unique positive root of $x^2-x-1$) is the least growth rate greater than 1.
Indeed, they prove a stronger statement, known as the \emph{Fibonacci dichotomy}, that if $|\CCC_n|$ is ever less
than the $n$th Fibonacci number, then $|\CCC_n|$ is eventually polynomial (see also~\cite{HV2006}).
Kaiser \& Klazar also determine that
2 is the least limit point in the set of growth rates.

Klazar~\cite{Klazar2004} considered the least growth rate admitting uncountably many permutation classes (which he denoted $\kappa$) and proved that $\kappa$ is at least 2 and is no greater than
approximately $2.33529$
(a root of a quintic). 
Vatter~\cite{Vatter2011} determined the exact value of $\kappa$ to be the unique real root of $x^3-2\+x^2-1$ (approximately $2.20557$) and completed the
characterisation of all growth rates up to $\kappa$, proving that there are uncountably many permutation classes with growth rate $\kappa$ and only countably many with growth rate less than $\kappa$.
The phase
transition at $\kappa$ also has enumerative ramifications: Albert, Ru\v{s}kuc \& Vatter~\cite{ARV2012} have shown that every permutation class with growth rate less than
$\kappa$ has a rational generating function.

Balogh, Bollob\'as \& Morris~\cite{BBM2006} extended Kaiser \& Klazar's work to the more
general setting of \emph{ordered graphs}.
An ordered graph is a graph with a linear order on its vertices.
To each permutation we associate an ordered graph.
The (ordered) \emph{graph} of a permutation $\sigma$ of length~$n$ has
vertex set $\{1,\ldots,n\}$
with an edge between vertices $i$ and $j$ 
if $i<j$ and $\sigma(i)>\sigma(j)$.
See the figures below for illustrations.
We use $G_\sigma$ to denote the graph of $\sigma$.
An \emph{induced ordered subgraph} of an ordered graph $G$ is an induced subgraph of $G$ that inherits
its vertex ordering.
A set of ordered graphs closed under taking induced ordered subgraphs is known as a \emph{hereditary class}.
It is easy to see that
$\tau\leqslant\sigma$
if and only if
$G_\tau$ is an induced ordered subgraph of $G_\sigma$.
Thus each permutation class is isomorphic to a hereditary class of ordered graphs, and results can be transferred between the two domains.

Balogh, Bollob\'as \& Morris conjectured that the set of growth
rates of hereditary classes of ordered graphs contains no limit points from above, and also that all
such growth rates are integers or algebraic irrationals.
These conjectures were
disproved by Albert \& Linton~\cite{AL2009} who exhibited an uncountable \emph{perfect set} (a closed set every member of which is a limit point) of growth rates of permutation classes and in turn conjectured that the set of growth rates includes some interval $(\lambda,\infty)$.\footnote{Commenting on this, Klazar~\cite{Klazar2010} states, ``It
seems that the refuted conjectures should have been phrased for finitely
based downsets.''.}

Their conjecture was established by Vatter~\cite{Vatter2010b} who proved that there are permutation classes of every growth rate at least $\lambda_A\approx2.48187$ (the unique real root of $x^5-2\+x^4-2\+x^2-2\+x-1$). In addition, he conjectured 
that this value was optimal, the set of growth rates below $\lambda_A$ being nowhere dense.
By generalising Vatter's constructions, we now refute his conjecture by proving the following two results:

\thmbox{
\begin{thmO}\label{thmTheta}
Let $\theta_B\approx2.35526$ be the unique real root of $x^7-2\+x^6-x^4-x^3-2\+x^2-2\+x-1$.
For any $\varepsilon>0$, there exist $\delta_1$ and $\delta_2$ with $0<\delta_1<\delta_2<\varepsilon$ such that every value in the interval $[\theta_B+\delta_1,\theta_B+\delta_2]$ is the growth rate of a permutation class.
\end{thmO}
}

\thmbox{
\begin{thmO}\label{thmLambda}
Let $\lambda_B\approx2.35698$ be the unique positive root of $x^8-2\+x^7-x^5-x^4-2\+x^3-2\+x^2-x-1$.
Every value at least $\lambda_B$ is the growth rate of a permutation class.
\end{thmO}
}
\vspace{6pt}

\newcommand{\grset}{{\boldsymbol\Gamma}}
In order to complete the characterisation of all permutation class growth rates, further investigation is required of the values between $\kappa$ and $\lambda_B$.
This interval contains the first limit point from above (which we denote $\xi$) and the first perfect set (whose infimum we denote $\eta$), as well as the first interval (whose infimum we denote $\theta$).
In~\cite{Vatter2010b},
Vatter showed that $\xi_A\approx2.30522$ (the unique real root of $x^5-2\+x^4-x^2-x-1$) is the infimum of a perfect set in the set of growth rates and hence a limit point from above, and conjectured that this is the least such limit point (in which case $\eta=\xi$).
If $\grset$ denotes the set of all permutation class growth rates, then the current state of our knowledge can be summarised by the following table:
\begin{center}
\renewcommand{\arraystretch}{1.25}
\begin{tabular}{|c|l|r|}
  \hline
  $\kappa$  & $\inf\,\{\+\gamma:|\grset\cap(1,\gamma)|>\aleph_0\+\}$ & $\kappa\approx2.20557$ \\ \hline
  $\xi$     & $\inf\,\{\+\gamma:\gamma\text{~is a limit point from above in~}\grset\+\}$ & $\xi\leqslant\xi_A\approx2.30522$ \\ \hline
  $\eta$    & $\inf\,\{\+\gamma:\gamma\in P, P\subset\grset, P\text{~is~a~perfect~set}\+\}$ & $\eta\leqslant\xi_A\approx2.30522$ \\ \hline
  $\theta$  & $\inf\,\{\+\gamma:(\gamma,\delta)\subset\grset\+\}$ & $\theta\leqslant\theta_B\approx2.35526$ \\ \hline
  $\lambda$ & $\inf\,\{\+\gamma:(\gamma,\infty)\subset\grset\+\}$ & $\lambda\leqslant\lambda_B\approx2.35698$ \\
  \hline
\end{tabular}
\end{center}
Further research is required to determine whether the upper bounds on $\xi$, $\eta$, $\theta$ and $\lambda$ are, in fact, equalities.
Currently, the only proven lower bound we have for any of these is $\kappa$.

Clearly, there are analogous phase transition values in the set of growth rates of hereditary classes of ordered graphs, and the permutation class upper bounds also bound the corresponding ordered graph values from above.
The techniques used in this paper 
could readily be applied directly to ordered graphs to yield smaller upper bounds
to the ordered graph phase transitions.
We leave such a study for the future.

\vspace{9pt}
In the next section, we investigate expansions of real numbers in non-integer bases.
In~\cite{Pedicini2005}, Pedi\-cini considered the case in which the digits are drawn from some allowable digit set
and determined the conditions on the base under which every real number in an interval can be expressed by such an expansion.
We prove two generalisations of Pedicini's result, firstly permitting each
digit to be drawn from a different allowable digit set, and secondly allowing the use of what we call \emph{generalised digits}.

In Section~\ref{sectSumClosed}, we consider the growth rates of \emph{sum-closed} permutation classes, whose members are constructed from a sequence of \emph{indecomposable} permutations.
Building on our results in the previous section concerning expansions of real numbers in non-integer bases, we
determine sufficient conditions for the set of growth rates of a family
of sum-closed permutation classes to include an interval of the real line.
These conditions relate to the enumeration of the indecomposable permutations in the classes. 

In the final section, we prove our two theorems by constructing families of permutation classes whose indecomposable permutations
are enumerated by sequences satisfying the required conditions.
The key permutations in our constructions are formed
from permutations known as \emph{increasing oscillations} by \emph{inflating} their ends.
Our constructions are similar to, but considerably more general than, those used by Vatter to prove that the set of permutation class growth rates includes the interval $[\lambda_A,\infty)$.
The proofs depend on some tedious calculations, performed using \emph{Mathematica}~\cite{Mathematica}. The details of these calculations are omitted from the proofs, but they can be found in~\cite{BevanLambdaCalcs}, which is a copy of the relevant
\emph{Mathematica} notebook.

\section{Non-integer bases and generalised digits}

Suppose we want a way of representing any value in an interval of the real line.
As is well known, infinite sequences of small non-negative integers suffice.
Given some integer $\beta>1$ and
any number
$x\in[0,1]$
there is some sequence $(a_n)$ where each
$a_n\in\{0,1,\ldots,\beta-1\}$
such that
$$
x \;=\;
\sum\limits_{n=1}^{\infty}a_n \+ \beta^{-n}
.
$$
This is simply the familiar ``base $\beta$'' expansion 
of $x$,
normally written $x=0.a_1a_2a_3\ldots$,
the $a_n$ being ``$\beta$-ary digits''.
For simplicity, we use $(a_n)_\beta$ to denote the sum above. 
For example,
if $a_n=1$ for each $n\geqslant1$
then $(a_n)_3=\thalf$.

In a seminal paper, R\'enyi~\cite{Renyi1957} generalized these expansions to arbitrary real
bases $\beta > 1$, and since then many papers have been devoted to the various connections with other branches of mathematics,
such as measure theory and dynamical systems. For further information, see the recent survey of Komornik~\cite{Komornik2011}.
Of particular relevance to us, Pedicini~\cite{Pedicini2005} explored the further generalisation in which the digits are drawn from some allowable digit set
$A = \{a_0,\ldots,a_m\}$,
with $a_0 < \ldots < a_m$.
He proved that every point in the interval
$\big[\frac{a_0}{\beta-1},\frac{a_m}{\beta-1}\big]$
has an expansion of the above form
with $a_n\in A$ for all $n\geqslant1$ if and only if
$
\max\limits_{1\leqslant j\leqslant m}(a_j-a_{j-1})
\leqslant
\frac{a_m-a_0}{\beta-1}
.
$

We establish two additional generalisations of Pedicini's result.
Firstly, we permit each digit to be drawn from a different
allowable digit set.
Specifically, 
for each~$n$
we allow the $n^{\text{th}}$
digit $a_n$ to be drawn from some finite set $D_n$ of permitted values.
For instance, we could have $D_n=\{1,4\}$ for odd $n$ and $D_n=\{1,3,5,7,9\}$ for even $n$.
We return to this example below.

Let $(D_n)_\beta=\{(a_n)_\beta:a_n\in D_n\}$ be the set of values for which there is an expansion in base~$\beta$.
The following lemma, generalising the result of Pedicini, gives sufficient conditions on $\beta$ for $(D_n)_\beta$ to be an interval of the real line.

\begin{lemma}\label{lemmaBaseNotation1}
Given a sequence $(D_n)$ of non-empty finite sets of non-negative real numbers, let $\ell_n$ and $u_n$
denote the least and greatest (`lower' and `upper') elements of $D_n$ respectively,
and let $\Delta_n=u_n-\ell_n$.
Also,
let
$\delta_n$
be the largest gap between consecutive elements of $D_n$ (with $\delta_n=0$ if $|D_n|=1$).

If $(u_n)$ is bounded, 
$\beta>1$,
and
for each $n$, $\beta$ satisfies the inequality
$$
\delta_n
 \;\leqslant\;
\sum\limits_{i=1}^{\infty}\Delta_{n+i}\+\beta^{-i},
$$
then
$(D_n)_\beta=\big[(\ell_n)_\beta,\,(u_n)_\beta\big]$.
\end{lemma}
\begin{proof}
If we let $A'_n=\{a-\ell_n:a\in D_n\}$ for each $n$, then the set of values that can be represented by digits from the $D_n$ is
the same as the set of values representable using digits from the $A'_n$ shifted 
by $(\ell_n)_\beta$:
$$
(D_n)_\beta \;=\; \{(a_n)_\beta:a_n\in D_n\} \;=\; \{(a'_n)_\beta+(\ell_n)_\beta:a'_n\in A'_n\}.
$$
So we need only consider cases in which each $\ell_n$ is zero, whence we have $\Delta_n=u_n$ for all $n$.

Given some $x$ in the specified interval,
we choose each digit maximally such that no partial sum exceeds $x$,
i.e. for each $n$, we select $a_n$ to be the greatest element of $D_n$ such that
$$
S(n) \;=\;
\sum\limits_{i=1}^{n}a_i\+ \beta^{-i} \;\leqslant\; x.
$$
It follows that
$(a_n)_\beta
=\liminfty S(n)
\leqslant x
$.
We claim that $(a_n)_\beta=x$.

If, for some $n$, $a_n<u_n$ then there is some element of $D_n$ greater than $a_n$ but no larger than $a_n+\delta_n$, so by our choice of the $a_n$, $S(n)+\delta_n\+\beta^{-n}>x$.
Hence, $|x-S(n)|<\delta_n\+\beta^{-n}$.
Thus, if $a_n<u_n$ for infinitely many values of $n$, $(a_n)_\beta=\liminfty S(n)=x$, since the $\delta_n$ are bounded and $\beta>1$.

Suppose now that $a_n=u_n$ for all but finitely many $n$.
If, in fact, $a_n=u_n$ for all $n$, then we have $x=(u_n)_\beta=(a_n)_\beta$ as required.
Indeed, this is the only possibility given the way we have chosen the $a_n$.\footnote{For example, our approach would select $0.50000\ldots$ rather than $0.49999\ldots$ as the decimal representation of $\half$, but $0.99999\ldots$ as the only possible representation for $1$.}
Let us assume that $N$ is the greatest index for which $a_N<u_N$, so that $S(N)+\delta_N\+\beta^{-N}>x$.

Now, by the constraints on $\beta$ and the fact that $\Delta_i=u_i$ for all $i$, we have
$
\delta_N\+ \beta^{-N}
\leqslant
\sum\limits_{i=N+1}^{\infty}u_i\+\beta^{-i},
$
and thus
$$
(a_n)_\beta \;=\; S(N) \:+\: \sum\limits_{i=N+1}^{\infty}u_i\+\beta^{-i} \;>\;x,
$$
which produces a contradiction.
\end{proof}
We refer to the conditions relating $\beta$ to the values of the $\delta_n$ and $\Delta_n$ as the \emph{gap inequalities}. These reoccur in subsequent lemmas.

Let us return to our example. If $D_n=\{1,4\}$ for odd $n$ and $D_n=\{1,3,5,7,9\}$ for even $n$, then to produce an interval of values it is sufficient for $\beta$ to satisfy the 
two gap inequalities:
$$
\begin{array}{rcl}
3 & \leqslant & 8\+\beta^{-1} + 3\+\beta^{-2} + 8\+\beta^{-3} + 3\+\beta^{-4} + \ldots \;=\; \dfrac{8\+\beta+3}{\beta^2-1},
\qquad\text{so~~} \beta\:\leqslant\: \frac{1}{3}\+(4+\sqrt{34}) \:\approx\: 3.27698 , \\[12pt]
2 & \leqslant & 3\+\beta^{-1} + 8\+\beta^{-2} + 3\+\beta^{-3} + 8\+\beta^{-4} + \ldots \;=\; \dfrac{3\+\beta+8}{\beta^2-1},
\qquad\text{so~~} \beta\:\leqslant\: \frac{1}{4}\+(3+\sqrt{89}) \:\approx\: 3.10850 .
\end{array}
$$
Thus, if $1<\beta\leqslant\frac{1}{4}\+(3+\sqrt{89})$ then $(D_n)_\beta$ 
is an interval.
For instance, with $\beta=3$, we have $(D_n)_3=\big[\frac{1}{2},\frac{21}{8}\big]$.

We now generalise our concept of a digit 
by permitting members of the $D_n$ to have the form $c_0.c_1\ldots c_k$ for non-negative integers $c_i$, where this is to be interpreted to mean
$$
c_0.c_1\ldots c_k \;=\; \sum\limits_{i=0}^{k}c_i\+ \beta^{-i} 
$$
as would be expected.
We call $c_0.c_1\ldots c_k$ a \emph{generalised digit} or a \emph{generalised $\beta$-digit} of \emph{length} $k+1$ with \emph{subdigits} $c_0,c_1,\ldots,c_k$.
Note that the value of a generalised $\beta$-digit depends on the value of~$\beta$.

As an example of sets of generalised digits,
consider
$D_n=\{1.1,1.11,1.12,1.2,1.21,1.22\}$ for odd~$n$ and $D_n=\{0\}$ for even $n$.
This is similar to what we use below to create intervals of permutation class growth rates. We return to this example below.

As long as the sets of generalised digits are suitably bounded, the same gap constraints on $\beta$ suffice as before to ensure that $(D_n)_\beta$ 
is an interval. Thus we have the following generalisation of Lemma~\ref{lemmaBaseNotation1}.

\begin{lemma}\label{lemmaBaseNotation2}
Given a sequence $(D_n)$ of non-empty finite sets of generalised $\beta$-digits, for a fixed $\beta$ let $\ell_n$ and $u_n$
denote the least and greatest elements of $D_n$ respectively,
and let $\Delta_n=u_n-\ell_n$.
Also,
let
$\delta_n$
be the largest gap between consecutive elements of $D_n$ (with $\delta_n=0$ if $|D_n|=1$).

If both the length and the subdigits of all the generalised digits in the $D_n$ are
bounded, 
$\beta>1$,
and
for each~$n$, $\beta$~satisfies the inequality
$$
\delta_n
 \;\leqslant\;
\sum\limits_{i=1}^{\infty}\Delta_{n+i}\+\beta^{-i},
$$
then
$(D_n)_\beta=\big[(\ell_n)_\beta,\,(u_n)_\beta\big]$.
\end{lemma}
We omit the proof since it is essentially the same as that for
Lemma~\ref{lemmaBaseNotation1}.

For our constructions, we only require the following special case.
\begin{cor}\label{corGapIneqs}
Let $k\geqslant3$ be odd. Suppose that for all odd $n>k$, we have $D_n=D_k$, and for all other $n>1$, we have $D_n=\{0\}$.
Then the gap inequalities for the sequence $(D_n)$ are
$$
(\beta^2-1)\+\delta_k \;\leqslant\; \Delta_k
\qquad
\text{and}
\qquad
(\beta^{k-1}-\beta^{k-3})\+\delta_1 \;\leqslant\; \Delta_k
.
$$
\end{cor}

Let us analyse our generalised digit example.
If $D_n=\{1.1,1.11,1.12,1.2,1.21,1.22\}$ for odd~$n$ and $D_n=\{0\}$ for even $n$,
then for odd $n$ we have $\Delta_n=0.12
=\beta^{-1}+2\+\beta^{-2}$.
However, the determination of $\delta_n$ depends on whether $1.2$ exceeds $1.12$ or not.
If it does (which is the case when $\beta$ exceeds~2), then $\delta_n$ is the greater of
$\beta^{-2}$
and
$\beta^{-1}-2\+\beta^{-2}$.
If $\beta<2$,
then $\delta_n$
is the greatest of $\beta^{-2}$, $\beta^{-1}-\beta^{-2}$ and $-\beta^{-1}+2\+\beta^{-2}$.
Careful solving of all the resulting inequalities reveals that
if $1<\beta\leqslant\half\+(1+\sqrt{13})\approx2.30278$,
then $(D_n)_\beta$ 
is an interval.\footnote
{When the $D_n$ contain generalised digits,
the solution set need not consist of a single interval.
For example, if
$D_n=\{0.5,0.501,0.502,0.51,0.511,0.512,0.52,0.521,1.3\}$ for odd~$n$
and
$D_n=\{0\}$ for even~$n$, then
$(D_n)_\beta$ 
is an interval if $1<\beta\leqslant2.732...$ and also if $2.879...\leqslant\beta\leqslant2.923...$. 
}

We are now ready to begin relating this to the construction of permutation classes.

\section{Sum-closed permutation classes}\label{sectSumClosed}

Given two permutations $\sigma$ and $\tau$ with lengths $k$ and $\ell$ respectively, their \emph{direct sum} $\sigma\oplus\tau$ is the permutation of length $k+\ell$ consisting of $\sigma$ followed by a shifted copy of $\tau$:
$$
(\sigma\oplus\tau)(i) \;=\;
\begin{cases}
  \sigma(i)   & \text{if~} i\leqslant k , \\
  k+\tau(i-k) & \text{if~} k+1 \leqslant i\leqslant k+\ell .
\end{cases}
$$
See Figure~\ref{figPermSums} for an illustration.

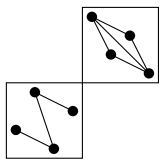
\begin{figure}[ht]
  $$
  \begin{tikzpicture}[scale=0.25]
    \plotpermnobox{}{2,4,1,3,8,6,7,5}
    \draw[] (0.5,0.5) rectangle (4.5,4.5);
    \draw[] (4.5,4.5) rectangle (8.5,8.5);
    \draw [] (1,2)--(3,1)--(2,4)--(4,3);
    \draw [] (5,8)--(6,6)--(8,5)--(7,7)--(5,8)--(8,5);
  \end{tikzpicture}
  \vspace{-3pt}
  $$
  \caption{The direct sum $\mathbf{2413}\oplus\mathbf{4231}$}
  \label{figPermSums}
\end{figure}
A permutation is called \emph{sum indecomposable} or just \emph{indecomposable} if it cannot be expressed as the direct sum of two shorter permutations.
For brevity, we call an indecomposable permutation simply an \emph{indecomposable}.
Observe that a permutation is indecomposable if and only if its (ordered) graph is connected.
Note that every permutation has a unique representation as the direct sum of a sequence of indecomposables. 

Suppose that $\SSS$ is a set of indecomposables 
that is downward closed under the subpermutation 
order (so if $\sigma\in\SSS$, $\tau\leqslant\sigma$ and $\tau$ is indecomposable, then $\tau\in\SSS$).
The \emph{sum closure} of $\SSS$, denoted~$\sumclosed\SSS$, is then the class of permutations of the form
$\sigma_1\oplus\sigma_2\oplus\ldots\oplus\sigma_r$, where each $\sigma_i\in\SSS$.
See Figure~\ref{figSumClosed} for an illustration.
Such a class is \emph{sum-closed}: if $\sigma,\tau\in\sumclosed\SSS$ then $\sigma\oplus\tau\in\sumclosed\SSS$.
Furthermore, every sum-closed class is the sum closure of its set of indecomposables.
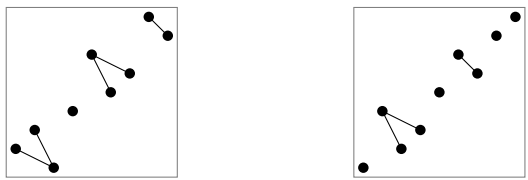
\begin{figure}[ht]
$$
        \begin{tikzpicture}[scale=0.25]
          \draw [] (1,2)--(3,1)--(2,3);
          \draw [] (6,5)--(5,7)--(7,6);
          \draw [] (8,9)--(9,8);
          \plotpermnobox{}{2,3,1,4,7,5,6,9,8}
          \draw[gray,thin] (0.5,0.5) rectangle (9.5,9.5);
        \end{tikzpicture}
        \qquad\qquad\qquad
        \begin{tikzpicture}[scale=0.25]
          \draw [] (3,2)--(2,4)--(4,3);
          \draw [] (6,7)--(7,6);
          \plotpermnobox{}{1,4,2,3,5,7,6,8,9}
          \draw[gray,thin] (0.5,0.5) rectangle (9.5,9.5);
        \end{tikzpicture}
        \vspace{-3pt}
$$
  \caption{Two permutations in $\bigoplus\!\big\{\mathbf{1},\mathbf{21},\mathbf{231},\mathbf{312}\big\}$}
  \label{figSumClosed}
\end{figure}

All the permutation classes that we construct to prove our results are sum-closed, and we define them by specifying the (downward closed)
sets of indecomposables 
from which they are composed.
The remainder of this section is devoted to study the growth rates of sum-closed classes, and our main results, Lemmas~\ref{lemmaGRInterval1} and~\ref{lemmaGRInterval2}, provide conditions on when the growth rates of a family of sum-closed permutation classes form an interval.
In Section~\ref{sectConstructions}, we construct families of sum-closed classes satisfying these conditions, thereby proving Theorems~\ref{thmTheta} and~ \ref{thmLambda}.

Let us say that a set of permutations $\SSS$ is \emph{positive} and \emph{bounded} (by a constant~$c$) if $1\leqslant|\SSS_n|\leqslant c$ for all $n$.
As long as its set of indecomposables is positive and bounded, the growth rate of a sum-closed permutation class is easy to determine.
Variants of the following lemma can be found in the work of Albert \& Linton~\cite{AL2009} and of Vatter~\cite{Vatter2010b}.
\begin{lemma}
\label{lemmaGRSumClosed}
If $\SSS$ is a downward closed set of indecomposables that is positive and bounded, then $\gr(\sumclosed\SSS)$ is the unique $\gamma$ such that $\gamma>1$ and $S(\gamma^{-1})=\sum\limits_{n=1}^\infty|\SSS_n|\+\gamma^{-n}=1$.
\end{lemma}
This follows from the following standard analytic combinatorial result.
Here, $\seq{\GGG}$ is the class of objects each element of which consists of a sequence of elements of $\GGG$.
\begin{lemma}
[{\cite[Theorem~V.1]{FS2009}} ``Asymptotics of supercritical sequences'']\label{lemmaSupercritical}
Suppose $\FFF=\seq{\GGG}$ 
and the corresponding generating functions, $F$ and $G$, are analytic at $0$, with $G(0) = 0$ and
$$
\lim\limits_{x\rightarrow\rho^-}G(x) \;>\; 1,
$$
where $\rho$ is the radius of convergence
of $G$. If, in addition, $G$ is aperiodic, i.e. there
does not exist an integer $d \geqslant 2$ such that $G(z) = H(z^d)$ for some $H$ analytic at $0$, then
$\gr(\FFF)=\sigma^{-1}$,
where $\sigma$ is the only root in $(0, \rho)$ of $G(x) = 1$.
\end{lemma}

\begin{proof}[Proof of Lemma~\ref{lemmaGRSumClosed}]
The sum-closed permutation class
$\sumclosed\SSS$ is in bijection with $\seq{\SSS}$.
Since $|\SSS_n|$ is positive and bounded, $S(z)$ is aperiodic, its radius of convergence is $1$,
and $\!\!\!\lim\limits_{\:\:\:x\rightarrow1^-}\!\!S(x)=+\infty$.
Thus
the requirements of Lemma~\ref{lemmaSupercritical} are satisfied, from which the desired result follows immediately.
\end{proof}


Building on our earlier results concerning expansions of real numbers
in non-integer bases,
we are now in a position to describe conditions under which
the set of growth rates of
a family of sum-closed permutation classes
includes an interval of the real line.
To state the lemma,
we introduce a notational convenience.
If $(a_n)$ is a bounded sequence of positive integers that enumerates a downward closed set of indecomposables $\SSS$ (i.e. $|\SSS_n|=a_n$ for all $n$), let us use
$\gr(\sumclosed (a_n))$ to denote the growth rate of
the sum-closed permutation class $\sumclosed\SSS$. 



\begin{lemma}\label{lemmaGRInterval1}
Given a sequence $(D_n)$ of non-empty finite sets of positive integers, let $\ell_n$ and $u_n$
denote the least and greatest elements of $D_n$ respectively,
and let $\Delta_n=u_n-\ell_n$.
Also,
let~$\delta_n$
be the largest gap between consecutive elements of $D_n$ (with $\delta_n=0$ if $|D_n|=1$).

If $(u_n)$ is bounded, then for every real number $\gamma$ such that $\gr(\sumclosed(\ell_n))\leqslant\gamma\leqslant\gr(\sumclosed(u_n))$ and
$$
\delta_n
 \;\leqslant\;
\sum\limits_{i=1}^{\infty}\Delta_{n+i}\+\gamma^{-i}
$$
for each~$n$,
there is some sequence $(a_n)$ with each
$a_n\in D_n$
such that
$\gamma=\gr(\sumclosed(a_n))$.
\end{lemma}
\begin{proof}
Let $\gamma_1=\gr(\sumclosed(\ell_n))$ and $\gamma_2=\gr(\sumclosed(u_n))$, so $\gamma_1\leqslant\gamma\leqslant\gamma_2$.
Now, by Lemma~\ref{lemmaGRSumClosed},
$$
(\ell_n)_\gamma
\;=\;
\sum\limits_{n=1}^{\infty}\ell_n\+\gamma^{-n}
\;\leqslant\;
\sum\limits_{n=1}^{\infty}\ell_n\+\gamma_1^{-n}
\;=\;
1
\;=\;
\sum\limits_{n=1}^{\infty}u_n\+\gamma_2^{-n}
\;\leqslant\;
\sum\limits_{n=1}^{\infty}u_n\+\gamma^{-n}
\;=\;
(u_n)_\gamma .
$$
Hence,
$
(\ell_n)_\gamma
\leqslant
1
\leqslant
(u_n)_\gamma
$
and so, since $\gamma$ satisfies the
gap inequalities,
by Lemma~\ref{lemmaBaseNotation1}
there is some sequence $(a_n)$ with each
$a_n\in D_n$
such that
$(a_n)_\gamma = 1$ and thus, by Lemma~\ref{lemmaGRSumClosed}, $\gamma=\gr(\sumclosed(a_n))$.
\end{proof}
This result generalises Proposition~2.4 in Vatter~\cite{Vatter2010b}, which treats only the case in which all the $D_n$ are identical and consist of an interval of integers $\{\ell,\ell+1,\ldots,u\}$.

As a consequence of this lemma, given a suitable sequence $(D_n)$ of sets of integers, if we could construct a family of sum-closed permutation classes such that
every sequence $(a_n)$ with each $a_n\in D_n$ enumerates
the indecomposables of
some member of the family, then the set of growth rates of the classes in the family would include an interval of the real line.
Returning to our original example in which $D_n=\{1,4\}$ for odd $n$ and $D_n=\{1,3,5,7,9\}$ for even $n$, suppose it were possible to construct
sum-closed classes whose indecomposables were enumerated by each sequence $\{(a_n):a_n\in D_n\}$.
Then the set of growth rates would include the interval from
$\gr(\sumclosed(1,1,\ldots))=2$ to the lesser of
$\gr(\sumclosed(4,9,4,9,\ldots))=2+\sqrt{14}$ 
and $\frac{1}{4}\+(3+\sqrt{89})$, the greatest value permitted by the gap inequalities. 
(Clearly, the fact that there are only countably many permutation classes with growth rate less than $\kappa\approx2.20557$ means that such a construction is, in fact, impossible.)

We now broaden this result to handle the case in which the $D_n$ contain generalised digits.
In the integer case, a term $a_n$ corresponds to a set consisting of $a_n$ indecomposables of length $n$.
In the generalised digit case, $a_n=c_0.c_1\ldots c_k$ corresponds to a set consisting of $c_0$ indecomposables of length $n$, $c_1$ indecomposables of length $n+1$, and so on, up to $c_k$ indecomposables of length $n+k$.

Before stating the generalised digit version of our lemma, we
generalise our terminology and notation.
Let us say that a sequence of generalised digits $(a_n)$ \emph{enumerates} a set~$\SSS$ of permutations if
$\SSS$ is the union $\biguplus \SSS^{\!(\!n\!)}$
of {disjoint} sets
$\SSS^{\!(\!n\!)}$,
such that, for each $n$, if $a_n=c_0.c_1\ldots c_k$ then $\SSS^{\!(\!n\!)}$~consists of exactly $c_i$ permutations of length $n+i$ for $i=0,\ldots,k$.
If $(a_n)$ is a sequence of generalised digits that enumerates a downward closed set of indecomposables $\SSS$,
let us use
$\gr(\sumclosed (a_n))$ to denote the growth rate
of the sum-closed permutation class $\sumclosed\SSS$.

In the proofs of our theorems, we make use of the following elementary fact.
Given a sequence of generalised digits that enumerates a set $\SSS$, there is a unique sequence of integers that also enumerates $\SSS$.
For example, if $a_n=1.12$ for odd $n$ and $a_n=3.1$ for even $n$, then
$(a_n)$ and the sequence
$(t_n)=(1,4,4,4,\ldots)$
are equivalent in this way.
We write such equivalences $(a_n)\equiv(t_n)$.

\begin{lemma}\label{lemmaGRInterval2}
Given a sequence $(D_n)$ of non-empty finite sets of generalised $\gamma$-digits, for a fixed $\gamma$ let $\ell_n$ and $u_n$
denote the least and greatest elements of $D_n$ respectively,
and let $\Delta_n=u_n-\ell_n$.
Also,
let
$\delta_n$
be the largest gap between consecutive elements of $D_n$ (with $\delta_n=0$ if $|D_n|=1$).

If the length and subdigits of all the generalised $\gamma$-digits in the $D_n$ are
bounded,
then for every real number $\gamma$ such that $\gr(\sumclosed(\ell_n))\leqslant\gamma\leqslant\gr(\sumclosed(u_n))$ and
$$
\delta_n
 \;\leqslant\;
\sum\limits_{i=1}^{\infty}\Delta_{n+i}\+\gamma^{-i}
$$
for each~$n$,
there is some sequence $(a_n)$ with each
$a_n\in D_n$
such that
$\gamma=\gr(\sumclosed(a_n))$.
\end{lemma}
We omit the proof since it is essentially the same as that for
Lemma~\ref{lemmaGRInterval1}, but making use of Lemma~\ref{lemmaBaseNotation2} rather than Lemma~\ref{lemmaBaseNotation1}.

Let us revisit our earlier generalised digit example in which
$D_n=\{1.1,1.11,1.12,1.2,1.21,1.22\}$ for odd~$n$ and $D_n=\{0\}$ for even $n$.
Suppose it were
possible to construct
sum-closed classes whose indecomposables were enumerated by each sequence $\{(a_n):a_n\in D_n\}$.
Then the set of growth rates of the classes
would include the interval from
$\gr(\sumclosed(1,1,\ldots))=2$ to the lesser of
$\gr(\sumclosed(1,2,3,2,3,\ldots))\approx2.51155$ (the real root of a cubic)
and $\half\+(1+\sqrt{13})$, the maximum value permitted by the gap inequalities.
(As with our previous example, such a construction is, in fact, impossible.)

We also need to use the following elementary analytic result, also found in Vatter~\cite{Vatter2010b}.
\begin{lemma}[{\cite[Proposition 2.3]{Vatter2010b}}] 
\label{lemmaGRSumClosedClose}
Given $\varepsilon>0$ and $c>0$, there is a positive integer $m$, dependent only on $\varepsilon$ and $c$, such that if
$(r_n)$ and $(s_n)$ are two sequences of positive integers bounded by $c$, and
$r_n=s_n$ for all $n\leqslant m$, then
$\gr(\sumclosed(r_n))$ and $\gr(\sumclosed(s_n))$ differ by no more than $\varepsilon$.
\end{lemma}
\begin{proof}
Let $\gamma_1=\gr(\sumclosed(r_n))$ and $\gamma_2=\gr(\sumclosed(s_n))$ and suppose that $\gamma_1<\gamma_2$.
Note that $\gamma_1\geqslant2$ because $(r_n)$ is positive, and $\gamma_2\leqslant c+1$ because $(s_n)$ is bounded by $c$.

From Lemma~\ref{lemmaGRSumClosed}, we have
$
\sum\limits_{n=1}^\infty r_n\+\gamma_1^{-n}
=
\sum\limits_{n=1}^\infty s_n\+\gamma_2^{-n}
$.
Hence, since
$r_n=s_n$ for $n\leqslant m$,
$$
\frac{\gamma_2-\gamma_1}{(c+1)^2}
\;\leqslant\;
\gamma_1^{-1}-\gamma_2^{-1}
\;\leqslant\;
\sum\limits_{n\leqslant m}
r_n\+(\gamma_1^{-n}-\gamma_2^{-n})
\;=\;
\sum\limits_{n>m} s_n\+\gamma_2^{-n} - \sum\limits_{n>m} r_n\+\gamma_1^{-n}
\;\leqslant\;
\frac{c}{2^m} .
$$
Thus, $m=\ceil{\log_2 
\frac{c\+(c+1)^2}{\varepsilon}
}$ suffices.
\end{proof}

\section{Constructions that yield intervals of growth rates}\label{sectConstructions}

To prove our theorems, we construct families of permutation classes whose indecomposables are enumerated by sequences
satisfying the requirements of Lemma~\ref{lemmaGRInterval2}.
In doing this, it helps to take a graphical perspective.
Recall from Section~\ref{sectLambdaPaperIntro} that the (ordered) graph, $G_\sigma$, of a permutation $\sigma$ of length~$n$ has
vertex set $\{1,\ldots,n\}$
with an edge between vertices $i$ and $j$ 
if $i<j$ and $\sigma(i)>\sigma(j)$.
Recall also that
$\tau\leqslant\sigma$
if and only if
$G_\tau$ is an induced ordered subgraph of $G_\sigma$ and that $\sigma$ is indecomposable if and only if $G_\sigma$ is \emph{connected}.
It tends to be 
advantageous to think of indecomposables as those permutations whose graphs are connected.

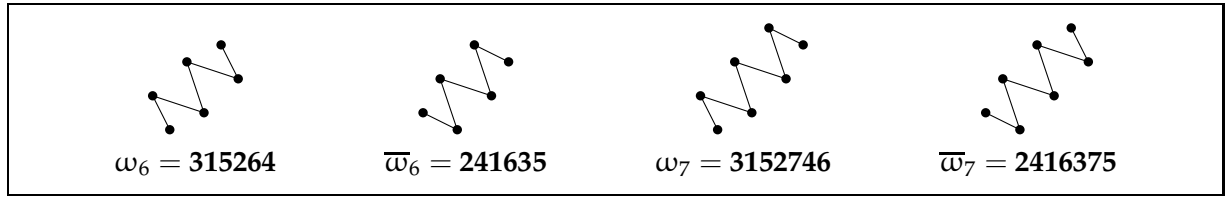
\begin{figure}[ht]
  \mybox{
  \vspace{3pt}
  $$
  \begin{tikzpicture}[scale=0.225]
    \plotpermnobox{6}{3, 1, 5, 2, 6, 4}
    \draw [thin] (2,1)--(1,3)--(4,2)--(3,5)--(6,4)--(5,6);
    \node at(3.5,-1){$\omega_6 = \mathbf{315264}$};
  \end{tikzpicture}
  \quad\qquad
  \begin{tikzpicture}[scale=0.225]
    \plotpermnobox{6}{2, 4, 1, 6, 3, 5}
    \draw [thin] (1,2)--(3,1)--(2,4)--(5,3)--(4,6)--(6,5);
    \node at(3.5,-1){$\overline\omega_6 = \mathbf{241635}$};
  \end{tikzpicture}
  \quad\qquad
  \begin{tikzpicture}[scale=0.225]
    \plotpermnobox{6}{3, 1, 5, 2, 7, 4, 6}
    \draw [thin] (2,1)--(1,3)--(4,2)--(3,5)--(6,4)--(5,7)--(7,6);
    \node at(3.5,-1){$\omega_7 = \mathbf{3152746}$};
  \end{tikzpicture}
  \quad\qquad
  \begin{tikzpicture}[scale=0.225]
    \plotpermnobox{6}{2, 4, 1, 6, 3, 7, 5}
    \draw [thin] (1,2)--(3,1)--(2,4)--(5,3)--(4,6)--(7,5)--(6,7);
    \node at(3.5,-1){$\overline\omega_7 = \mathbf{2416375}$};
  \end{tikzpicture}
  \vspace{-6pt}
  $$
  }
  \caption{Primary and secondary oscillations}\label{figOscillating}
\end{figure}
All our constructions are based on permutations whose graphs are \emph{paths}.
Such permutations are called \emph{increasing oscillations}. We distinguish two cases,
the \emph{primary oscillation} of length $n$, which we denote $\omega_n$,
whose \emph{least} entry corresponds to an end of the path,
and the \emph{secondary oscillation} of length $n$, denoted~$\overline\omega_n$
whose \emph{first} entry corresponds to an end of the path.
See Figure~\ref{figOscillating} for an illustration.
Trivially, we have $\omega_1=\overline\omega_1=\mathbf{1}$ and $\omega_2=\overline\omega_2=\mathbf{21}$.
The (lower and upper) \emph{ends} of an increasing oscillation are the
entries corresponding to the (lower and upper)
ends
of its path graph.

\HIDE{
Formally, for $n\geqslant 4$, the primary oscillation of length $n$, $\omega_n$, and the secondary oscillation of length $n$, $\overline\omega_n$, can be defined as follows:
$$
\begin{array}{rclr}
\omega_n          &=&          3,    1,\,\,\, 5,2,\,\,\, 7,4,\,\,\, 9,6,\,\,\, \ldots,\,\,\, n-3,n-6,\,\,\, n-1,            n-4,\,\,\,
         n,   n-2 & \text{~for even~}n\geqslant4, \\[3pt]
\omega_n          &=&          3,    1,\,\,\, 5,2,\,\,\, 7,4,\,\,\, 9,6,\,\,\, \ldots,\,\,\, n-2,n-5,\,\,\, n,\phantom{{}-0}n-3,\,\,\, \phantom{n,{}}n-1 & \text{~for odd~}n\geqslant5,  \\[3pt]
\overline\omega_n &=& \phantom{0,{}} 2,\,\,\, 4,1,\,\,\, 6,3,\,\,\, 8,5,\,\,\, \ldots,\,\,\, n-2,n-5,\,\,\, n,\phantom{{}-0}n-3,\,\,\, \phantom{n,{}}n-1 & \text{~for even~}n\geqslant4, \\[3pt]
\overline\omega_n &=& \phantom{0,{}} 2,\,\,\, 4,1,\,\,\, 6,3,\,\,\, 8,5,\,\,\, \ldots,\,\,\, n-3,n-6,\,\,\, n-1,            n-4,\,\,\,
         n,   n-2 & \text{~for odd~}n\geqslant5.
\end{array}
$$
} 

The key permutations in our constructions are 
formed by \emph{inflating} (i.e. replacing)
each end of a primary oscillation of odd length with an increasing permutation.
For $n\geqslant4$ and $r,s\geqslant1$,
let~$\omega_n^{r,s}$ denote the
permutation of length $n-2+r+s$
formed by replacing 
the lower and upper ends of $\omega_n$ with the increasing permutations of length $r$ and~$s$ respectively.
The graph of $\omega_n^{r,s}$ thus consists of a path on $n-2$ vertices with $r$ pendant edges attached to its lower end and
$s$ pendant edges attached to its upper end.
See Figure~\ref{figR743} for some examples.
We also occasionally make use of $\overline\omega_n^{r,s}$, which we define analogously.

\begin{figure}[ht]
  \mybox{
  \vspace{-6pt}
  $$
  \begin{tikzpicture}[scale=0.225]
    \plotpermnobox{9}{4, 1, 2, 6, 3, 9, 5, 7, 8}
    \draw [thin] (2,1)--(1,4);
    \draw [thin] (3,2)--(1,4)--(5,3)--(4,6)--(7,5)--(6,9)--(8,7);
    \draw [thin] (6,9)--(9,8);
    \node at(5,-1){$\;\;\omega_7^{2,2}$};
  \end{tikzpicture}
  \quad
  \begin{tikzpicture}[scale=0.225]
    \plotpermnobox{10}{4, 1, 2, 6, 3, 10, 5, 7, 8, 9}
    \draw [thin] (2,1)--(1,4);
    \draw [thin] (3,2)--(1,4)--(5,3)--(4,6)--(7,5)--(6,10)--(8,7);
    \draw [thin] (10,9)--(6,10)--(9,8);
    \node at(5.5,-1){$\;\;\omega_7^{2,3}$};
  \end{tikzpicture}
  \quad
  \begin{tikzpicture}[scale=0.225]
    \plotpermnobox{10}{5, 1, 2, 3, 7, 4, 10, 6, 8, 9}
    \draw [thin] (2,1)--(1,5)--(4,3);
    \draw [thin] (3,2)--(1,5)--(6,4)--(5,7)--(8,6)--(7,10)--(9,8);
    \draw [thin] (7,10)--(10,9);
    \node at(5.5,-1){$\;\;\omega_7^{3,2}$};
  \end{tikzpicture}
  \quad
  \begin{tikzpicture}[scale=0.225]
    \plotpermnobox{11}{5, 1, 2, 3, 7, 4, 11, 6, 8, 9, 10}
    \draw [thin] (2,1)--(1,5)--(4,3);
    \draw [thin] (3,2)--(1,5)--(6,4)--(5,7)--(8,6)--(7,11)--(9,8);
    \draw [thin] (11,10)--(7,11)--(10,9);
    \node at(6,-1){$\;\;\omega_7^{3,3}$};
  \end{tikzpicture}
  \quad
  \begin{tikzpicture}[scale=0.225]
    \plotpermnobox{11}{6, 1, 2, 3, 4, 8, 5, 11, 7, 9, 10}
    \draw [thin] (2,1)--(1,6)--(4,3);
    \draw [thin] (5,4)--(1,6);
    \draw [thin] (3,2)--(1,6)--(7,5)--(6,8)--(9,7)--(8,11)--(11,10);
    \draw [thin] (8,11)--(10,9);
    \node at(6,-1){$\;\;\omega_7^{4,2}$};
  \end{tikzpicture}
  \quad
  \begin{tikzpicture}[scale=0.225]
    \plotpermnobox{12}{6, 1, 2, 3, 4, 8, 5, 12, 7, 9, 10, 11}
    \draw [thin] (2,1)--(1,6)--(4,3);
    \draw [thin] (5,4)--(1,6);
    \draw [thin] (3,2)--(1,6)--(7,5)--(6,8)--(9,7)--(8,12)--(11,10);
    \draw [thin] (12,11)--(8,12)--(10,9);
    \node at(6.5,-1){$\;\;\omega_7^{4,3}$};
  \end{tikzpicture}
  \vspace{-9pt}
  $$
  }
  \caption{The set of permutations $R_7^{4,3}$}\label{figR743}
\end{figure}
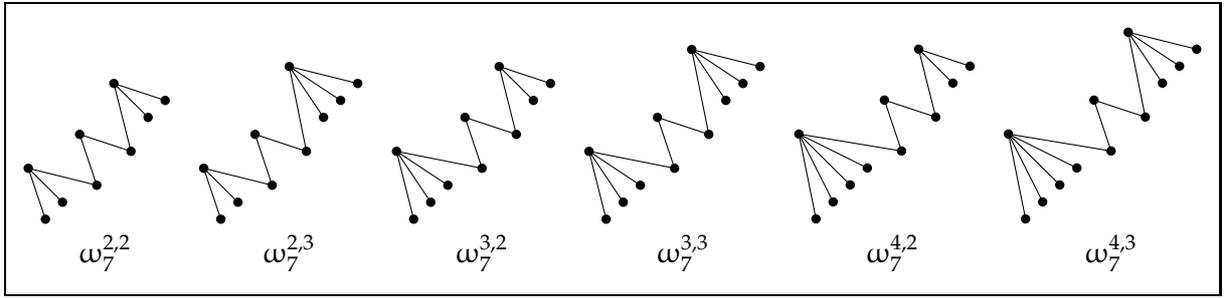
The building blocks of our constructions are the following sets of indecomposables.
Given a primary oscillation with both ends inflated, i.e. some $\omega_n^{r,s}$ with $r,s\geqslant2$, we define~$R_n^{r,s}$ to be the set of primary oscillations with both ends inflated that are
subpermutations of~$\omega_n^{r,s}$:
$$
R_n^{r,s} \;=\; \{\omega_n^{u,v} \,:\, 2\leqslant u\leqslant r,\,2\leqslant v\leqslant s\}.
$$
We only make use of $R_n^{r,s}$ for odd $n$. 
See Figure~\ref{figR743} for an example.

Let us investigate the properties of these sets.
Firstly, if $n\neq m$, then no element of $R_n^{r,s}$ is a subpermutation of 
an element of~$R_m^{r,s}$.
This is a consequence of 
the following elementary observation,
which follows directly from consideration of the (ordered) graphs of the permutations. 
\begin{obs} 
If $n,m\geqslant4$, $n\neq m$ and $r,s,u,v\geqslant2$, then $\omega_n^{r,s}$ and $\omega_m^{u,v}$ are incomparable under the subpermutation 
order.
\end{obs}

Thus for fixed $r$ and $s$, and varying $n$, the $R_n^{r,s}$ form a collection of sets of permutations, no member of one being a subpermutation of a member of another. This 
should be compared with
the 
concept of an \emph{antichain}, which is a set of permutations, none of which is a subpermutation of another.

We build permutation classes by specifying that, for some fixed $r$ and $s$, their indecomposables must
include some subset of $R_n^{r,s}$ for each (odd) $n$.
Because, if $n\neq m$, any element of $R_n^{r,s}$ is incomparable with any element of $R_m^{r,s}$,
the choice of subset to include can be made independently for each~$n$.
This provides the flexibility we need to construct families of classes whose growth rates include an interval.
(Vatter's simpler constructions in~\cite{Vatter2010b},
which are also based on inflating the ends of increasing oscillations,
rely on being able to choose any subset of an infinite antichain to include in the indecomposables of a class.)

Clearly, for each $n$, the subset of $R_n^{r,s}$ included in the indecomposables must be downward closed.
It has been observed heuristically that
subsets containing $\omega_n^{3,2}$ 
are better for generating intervals of growth rates.
In light of this, we define $\FFF_n^{r,s}$ to be the family of downward closed subsets of $R_n^{r,s}$ that contain $\omega_n^{3,2}$ (and hence also contain $\omega_n^{2,2}$).
See Figure~\ref{figRHasseDiag} for an illustration.

For example, $\FFF_n^{4,3}$ consists of the seven downsets whose 
sets of maximal elements
are
$$
\{\omega_n^{3,2}\},\;\; \{\omega_n^{3,2},\omega_n^{2,3}\},\;\; \{\omega_n^{3,3}\},\;\; \{\omega_n^{4,2}\},\;\; \{\omega_n^{4,2},\omega_n^{2,3}\},\;\; \{\omega_n^{4,2},\omega_n^{3,3}\}, \;\;\text{and}\;\; \{\omega_n^{4,3}\}.
$$
To each downset $S$ in $\FFF_n^{r,s}$, we can associate a generalised digit $c_0.c_1\ldots c_k$ 
such that $c_i$ records 
the number of elements in $S$ of length $n+2+i$ for each $i\geqslant0$.
For example, the generalised digits that enumerate 
the elements of $\FFF_n^{4,3}$ are (in the same order as above)
$$
1.1,\;\; 1.2,\;\; 1.21,\;\; 1.11,\;\; 1.21,\;\; 1.22 \;\;\text{and}\;\; 1.221.
$$
Note that distinct downsets in $\FFF_n^{r,s}$ may be enumerated by the same generalised digit.

\begin{figure}[ht]
\newcommand{\hpt}[4]{\fill[white,radius=0.425] (#1,#2) circle; \node at(#1,#2){\scriptsize{#3,#4}};}
\newcommand{\hpg}[4]{\fill[white,radius=0.425] (#1,#2) circle; \node [gray] at(#1,#2){\scriptsize{#3,#4}};}
$$
  \begin{tikzpicture}[scale=0.525]
    \draw [thin] ( 0,4)--(-5, 9);
    \draw [thin] ( 1,5)--(-4,10);
    \draw [thin] ( 2,6)--(-3,11);
    \draw [thin] ( 3,7)--(-2,12);
    \draw [thin] ( 4,8)--(-1,13);
    \draw [thin] ( 0,4)--( 4, 8);
    \draw [thin] (-1,5)--( 3, 9);
    \draw [thin] (-2,6)--( 2,10);
    \draw [thin] (-3,7)--( 1,11);
    \draw [thin] (-4,8)--( 0,12);
    \draw [thin] (-5,9)--(-1,13);
    \draw [very thick] (-6,9)--(-5,10)--(-3,8)--(0,11)--(3,8)--(4,9)--(5,8);
    \hpt{ 0}{ 4}{2}{2} \hpt{-1}{ 5}{3}{2} \hpt{-2}{ 6}{4}{2} \hpt{-3}{ 7}{5}{2} \hpt{-4}{ 8}{6}{2} \hpt{-5}{ 9}{7}{2}
    \hpt{ 1}{ 5}{2}{3} \hpt{ 0}{ 6}{3}{3} \hpt{-1}{ 7}{4}{3} \hpt{-2}{ 8}{5}{3} \hpg{-3}{ 9}{6}{3} \hpg{-4}{10}{7}{3}
    \hpt{ 2}{ 6}{2}{4} \hpt{ 1}{ 7}{3}{4} \hpt{ 0}{ 8}{4}{4} \hpt{-1}{ 9}{5}{4} \hpg{-2}{10}{6}{4} \hpg{-3}{11}{7}{4}
    \hpt{ 3}{ 7}{2}{5} \hpt{ 2}{ 8}{3}{5} \hpt{ 1}{ 9}{4}{5} \hpt{ 0}{10}{5}{5} \hpg{-1}{11}{6}{5} \hpg{-2}{12}{7}{5}
    \hpt{ 4}{ 8}{2}{6} \hpg{ 3}{ 9}{3}{6} \hpg{ 2}{10}{4}{6} \hpg{ 1}{11}{5}{6} \hpg{ 0}{12}{6}{6} \hpg{-1}{13}{7}{6}
  \end{tikzpicture}
\qquad\qquad
  \begin{tikzpicture}[scale=0.525]
    \draw [thin] ( 0,4)--(-5, 9);
    \draw [thin] ( 1,5)--(-4,10);
    \draw [thin] ( 2,6)--(-3,11);
    \draw [thin] ( 3,7)--(-2,12);
    \draw [thin] ( 4,8)--(-1,13);
    \draw [thin] ( 0,4)--( 4, 8);
    \draw [thin] (-1,5)--( 3, 9);
    \draw [thin] (-2,6)--( 2,10);
    \draw [thin] (-3,7)--( 1,11);
    \draw [thin] (-4,8)--( 0,12);
    \draw [thin] (-5,9)--(-1,13);
    \draw [very thick] (-6,9)--(-4,11)--(-2,9)--(-1,10)--(1,8)--(2,9)--(3,8)--(4,9)--(5,8);
    \hpt{ 0}{ 4}{2}{2} \hpt{-1}{ 5}{3}{2} \hpt{-2}{ 6}{4}{2} \hpt{-3}{ 7}{5}{2} \hpt{-4}{ 8}{6}{2} \hpt{-5}{ 9}{7}{2}
    \hpt{ 1}{ 5}{2}{3} \hpt{ 0}{ 6}{3}{3} \hpt{-1}{ 7}{4}{3} \hpt{-2}{ 8}{5}{3} \hpt{-3}{ 9}{6}{3} \hpt{-4}{10}{7}{3}
    \hpt{ 2}{ 6}{2}{4} \hpt{ 1}{ 7}{3}{4} \hpt{ 0}{ 8}{4}{4} \hpt{-1}{ 9}{5}{4} \hpg{-2}{10}{6}{4} \hpg{-3}{11}{7}{4}
    \hpt{ 3}{ 7}{2}{5} \hpt{ 2}{ 8}{3}{5} \hpg{ 1}{ 9}{4}{5} \hpg{ 0}{10}{5}{5} \hpg{-1}{11}{6}{5} \hpg{-2}{12}{7}{5}
    \hpt{ 4}{ 8}{2}{6} \hpg{ 3}{ 9}{3}{6} \hpg{ 2}{10}{4}{6} \hpg{ 1}{11}{5}{6} \hpg{ 0}{12}{6}{6} \hpg{-1}{13}{7}{6}
  \end{tikzpicture}
  \vspace{-4pt}
$$
  \caption{Two downsets in $\FFF_n^{7,6}$, both enumerated by the generalised digit $1.234531$, and the corresponding paths through the Hasse diagram of $R_n^{7,6}$}
  \label{figRHasseDiag}
\end{figure}
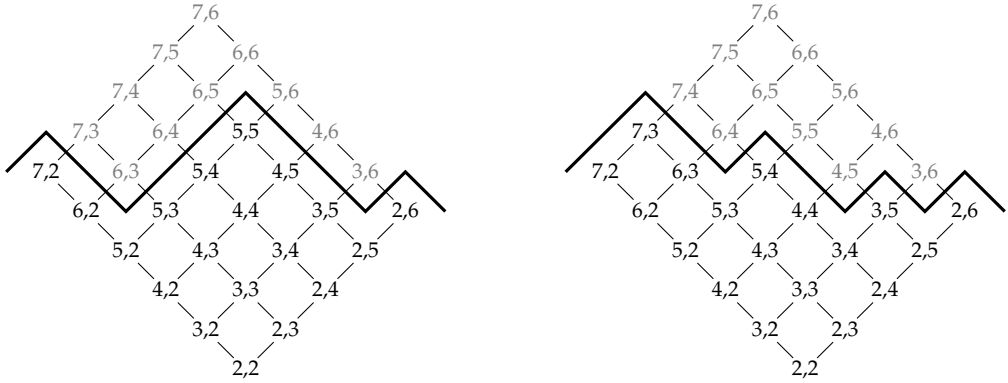
\begin{lemma}\label{lemmaSizeOfF}
The set $\FFF_n^{r,s}$ contains $\binom{r+s-2}{r-1}-s$ downsets with $\sum\limits_{i=1}^{s-1}(s-i)\binom{r-2}{i}-2$ distinct enumerations.
\end{lemma}
\begin{proof}
Each downset in $\FFF_n^{r,s}$ corresponds to a path consisting of northeast and southeast steps through the Hasse diagram of $R_n^{r,s}$ starting at the left of $\omega_n^{r,2}$ and ending at the right of $\omega_n^{2,s}$.
(see Figure~\ref{figRHasseDiag}). There are a total of $\binom{r+s-2}{r-1}$ such paths, $s$ of which correspond to downsets not containing $\omega_n^{3,2}$.

Given multiple downsets with the same enumeration,
we choose the one
for which, for each $k\geqslant4$, the $c_k$ elements of length $k$ are
$$
\omega_n^{k-2,2},\,\omega_n^{k-3,3},\,\ldots,\,\omega_n^{k-c_k-1,c_k+1}.
$$
This is the downset whose elements are
as far to the left as possible in the Hasse diagram of~$R_n^{r,s}$ (see the right diagram in Figure~\ref{figRHasseDiag}).
Such downsets correspond to paths in which a sequence of more than one northeast step may only occur initially or finally.
There are $\binom{r-2}{s-i-j-1}$ such paths with an initial sequence of $i$ northeast steps and a final sequence of $j$ northeast steps.
The result follows by summing over $i$ and $j$ and excluding the 
terms that correspond to
the empty downset and the singleton downset $\{\omega_n^{2,2}\}$.
\end{proof}

\Needspace*{5\baselineskip}
We also need to take into account the indecomposables that are subpermutations of 
$\omega_n^{r,s}$ but are \emph{not} elements of~$R_n^{r,s}$.
For odd $n$, these are of the following types:
\begin{bullets}
  \item Primary oscillations $\omega_m$ 
  and secondary oscillations $\overline\omega_m$.
  \item Permutations whose graphs are stars 
  $K_{1,u}$ 
  and whose first (and greatest) entry corresponds to the internal vertex;
  we use~$\psi_u$ to denote these star permutations.
  \item Increasing oscillations with just one end inflated: $\omega_m^{u,1}$, $\omega_m^{1,v}$ and $\overline\omega_m^{1,v}$. 
\end{bullets}
\vspace{6pt}
\begin{figure}[ht]
  \mybox{
  \vspace{-6pt}
  $$
  \begin{tikzpicture}[scale=0.225]
    \plotpermnobox{6}{3, 1, 5, 2, 6, 4}
    \draw [thin] (2,1)--(1,3)--(4,2)--(3,5)--(6,4)--(5,6);
    \node at(3.5,-1){$\;\;\omega_6$};
  \end{tikzpicture}
  \;\qquad\;
  \begin{tikzpicture}[scale=0.225]
    \plotpermnobox{6}{2, 4, 1, 6, 3, 5}
    \draw [thin] (1,2)--(3,1)--(2,4)--(5,3)--(4,6)--(6,5);
    \node at(3.5,-1){$\;\;\overline\omega_6$};
  \end{tikzpicture}
  \;\qquad\;
  \begin{tikzpicture}[scale=0.225]
    \plotpermnobox{6}{6, 1, 2, 3, 4, 5}
    \draw [thin] (2,1)--(1,6)--(3,2);
    \draw [thin] (4,3)--(1,6)--(5,4);
    \draw [thin] (6,5)--(1,6);
    \node at(3.5,-1){$\;\;\psi_5$}; 
  \end{tikzpicture}
  \;\qquad\;
  \begin{tikzpicture}[scale=0.225]
    \plotpermnobox{6}{5, 1, 2, 3, 6, 4}
    \draw [thin] (2,1)--(1,5)--(3,2);
    \draw [thin] (4,3)--(1,5)--(6,4)--(5,6);
    \node at(3.5,-1){$\;\;\omega_4^{3,1}$};
  \end{tikzpicture}
  \;\qquad\;
  \begin{tikzpicture}[scale=0.225]
    \plotpermnobox{6}{4, 1, 2, 6, 3, 5}
    \draw [thin] (2,1)--(1,4)--(3,2);
    \draw [thin] (1,4)--(5,3)--(4,6)--(6,5);
    \node at(3.5,-1){$\;\;\omega_5^{2,1}$};
  \end{tikzpicture}
  \;\qquad\;
  \begin{tikzpicture}[scale=0.225]
    \plotpermnobox{6}{3, 1, 6, 2, 4, 5}
    \draw [thin] (2,1)--(1,3)--(4,2)--(3,6)--(5,4);
    \draw [thin] (3,6)--(6,5);
    \node at(3.5,-1){$\;\;\omega_5^{1,2}$};
  \end{tikzpicture}
  \;\qquad\;
  \begin{tikzpicture}[scale=0.225]
    \plotpermnobox{6}{2,6,1,3,4,5}
    \draw [thin] (1,2)--(3,1)--(2,6)--(4,3);
    \draw [thin] (6,5)--(2,6)--(5,4);
    \node at(3.5,-1){$\;\;\overline\omega_4^{1,3}$};
  \end{tikzpicture}
  \vspace{-9pt}
  $$
  }
  \caption{The elements of $Q^{4,3}$ of size 6}\label{figS426}
\end{figure}
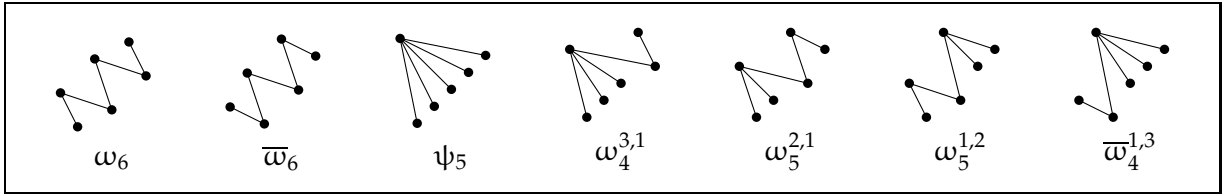
Given $r,s\geqslant2$, let $Q^{r,s}$ be the (infinite) set of 
indecomposables that are subpermutations of 
$\omega_n^{r,s}$ for some odd $n\geqslant5$, but are not elements of~$R_n^{r,s}$.
See Figure~\ref{figS426} for an illustration.

We can enumerate $Q^{r,s}$ explicitly. In doing so, we make use of two additional notational conventions. A sequence of integers $(a_n)$ whose entries have the same value for all $n\geqslant k$ is denoted $(a_1,a_2,\ldots,a_{k-1},\overline{a_k})$.
For example, 
$(1,\overline4)=(1,4,4,4,\ldots)$.
We also use
$c^k$, in both sequences and generalised digits, for the term $c$ repeated $k$ times.
Thus,
$(0^3,1^4)=(0,0,0,1,1,1,1)$ and $0.0^41= 0.00001$.

\begin{lemma}\label{lemmaQCount}
  Suppose $r\geqslant s$ and $(q_n)$ is the sequence that enumerates $Q^{r,s}$. Then,
  $$
  (q_n) \;=\;
  \begin{cases}
     \;
     \big(1,\,1,\,2,\,\overbrace{3,\,5,\,7,\,\ldots,\,2r-1}^{r-1},\,\overline{2r}\big)
     , & \text{if $r=s$;} \\[8pt]
     \;
     \overunderbraces
     {&\br{2}{s}} 
     {\big(1,\,1,\,2,\, & 3,\,5,\,7,\,\ldots,\, & 2s+1 & ,\,2s+2,\,2s+3,\,\ldots,\,\overline{r+s} & \big)}
     {&&\br{2}{r-s}} 
     , & \text{otherwise.}
  \end{cases}
  $$
\end{lemma}
\begin{proof}
  The set $Q^{r,s}$ is composed of the following disjoint sets of indecomposables:
  \begin{bullets}
    \item Primary oscillations $\omega_m$ for each $m\geqslant1$, enumerated by $(\overline{1})$.
    \item Secondary oscillations $\overline\omega_m$, for each $m\geqslant3$, enumerated by $(0^2,\overline{1})$.
    \item Star permutations $\psi_u$, for $3\leqslant u\leqslant r+1$, enumerated by $(0^3,1^{r-1})$.
    \item Primary oscillations with the lower end inflated $\omega_m^{u,1}$, for each $m\geqslant 4$ and $2\leqslant u\leqslant r$, enumerated by $(0^4,1,2,3,\ldots,r-2,\overline{r-1})$.
    \item Increasing oscillations with the upper end inflated, $\omega_m^{1,v}$ for odd $m$ and $\overline\omega_m^{1,v}$ for even $m$, for each $m\geqslant 4$ and $2\leqslant v\leqslant s$, enumerated by $(0^4,1,2,3,\ldots,s-2,\overline{s-1})$.
  \end{bullets}
  The result follows by adding the five sequences termwise.
\end{proof}

We now have all the building blocks we need.
Given $r\geqslant3$, $s\geqslant2$, and odd $k\geqslant5$, we define $\Phi_{r,s,k}$ to be the family of those permutation classes whose indecomposables are the union of $Q^{r,s}$ together with 
an element of $\FFF_n^{r,s}$ for each odd $n\geqslant k$:
$$
\Phi_{r,s,k} \;=\;
\Big\{
\sumclosed \big( Q^{r,s} \:\cup\: S_k \:\cup\: S_{k+2} \:\cup\: S_{k+4} \:\cup\: \ldots
\big) \::\:
S_n\in \FFF_n^{r,s}
, \, n=k,k+2,\ldots
\Big\} .
$$
The sequence of families $(\Phi_{5,3,k})_{k=5,7,\ldots}$ is what we need to prove our first theorem, which we restate below.
The set of growth rates of permutation classes in
each family $\Phi_{5,3,k}$ %
consists of an interval, $I_k$, such that the sequence $(I_k)$ of these intervals approaches $\theta_B$ from above.

\begin{repthm}{thmTheta}
Let $\theta_B\approx2.35526$ be the unique real root of $x^7-2\+x^6-x^4-x^3-2\+x^2-2\+x-1$.
For any $\varepsilon>0$, there exist $\delta_1$ and $\delta_2$ with $0<\delta_1<\delta_2<\varepsilon$ such that every value in the interval $[\theta_B+\delta_1,\theta_B+\delta_2]$ is the growth rate of a permutation class.
\end{repthm}
\begin{proof}
Let $(q_n)$ be the sequence that enumerates 
$Q^{5,3}$.
By Lemma~\ref{lemmaQCount},
$(q_n)=(1,1,2,3,5,7,\overline8)$. It can readily be checked using Lemma~\ref{lemmaGRSumClosed} that
$\gr(\sumclosed(q_n))=\theta_B$.

Let $k\geqslant5$ be odd.
For each odd $n\geqslant k+2$, let
$$
F_n \;=\; \{1.1,1.11,1.111,1.2,1.21,1.211,1.22,1.221,1.222,1.2221\}
$$
be the set of generalised digits
that enumerate 
sets of indecomposables in
$\FFF_{n-2}^{5,3}$.
Otherwise (if $n$ is even or $n\leqslant k$), let $F_n=\{\+0\+\}$.
Now, for each $n$, let $D_n=\{q_n+f:f\in F_n\}$.

So, by construction, for every permutation class $\sumclosed\SSS\in\Phi_{5,3,k}$ there is a corresponding sequence $(a_n)$, with each
$a_n\in D_n$, that enumerates $\SSS$.

Let $\ell_n=q_n+1.1$ for odd $n\geqslant k+2$ and $\ell_n=q_n$ otherwise.
Similarly, let $u_n=q_n+1.2221$ for odd $n\geqslant k+2$ and $u_n=q_n$ otherwise.
Note that $\ell_n$ and $u_n$ depend on $k$.
We have the following equivalences:
$$
\begin{array}{rcl}
(\ell_n) &\!\equiv\!& (1,1,2,3,5,7,\overbrace{8,8,\ldots,8}^{k-5},\overline9) \\[8pt]
(u_n)    &\!\equiv\!& (1,1,2,3,5,7,\overbrace{8,8,\ldots,8}^{k-5},9,10,11,\overline{12})
\end{array}
$$
We now apply Lemma~\ref{lemmaGRInterval2}.
By Corollary~\ref{corGapIneqs}, the gap inequalities are
$$
\begin{array}{rcl}
1                           -\gamma^{-2}              &\leqslant& \gamma^{-1} +2\+\gamma^{-2} +2\+\gamma^{-3} +\gamma^{-4} , \qquad \\[4pt]
  \gamma -1 -2\+\gamma^{-1} +\gamma^{-2} +\gamma^{-3} &\leqslant& \gamma^{-1} +2\+\gamma^{-2} +2\+\gamma^{-3} +\gamma^{-4} .
\end{array}
$$
These
necessitate only that 
the growth rate does not exceed $\gamma_{\max}\approx2.470979$, the unique positive root of $x^4-2\+x^3-x^2-1$. This is independent of the value of $k$, and
is
greater than $\gr(\sumclosed(u_n))$ for all odd $k\geqslant5$
since, by Lemma~\ref{lemmaGRSumClosed}, $\gr(\sumclosed(u_n))\approx2.362008$ if $k=5$.

So, for each $k$ we have an interval of growth rates:
If $\gamma$ is such that $\gr(\sumclosed(\ell_n))\leqslant\gamma\leqslant\gr(\sumclosed(u_n))$, then there is some permutation class in $\Phi_{5,3,k}$ whose growth rate is $\gamma$.

Moreover, by Lemma~\ref{lemmaGRSumClosedClose},
\[
\liminfty[k] \gr(\sumclosed(\ell_n)) \;=\;
\liminfty[k] \gr(\sumclosed(u_n)) \;=\;
\gr(\sumclosed(q_n)) \;=\; \theta_B, 
\]
so these intervals can be found arbitrarily close to $\theta_B$.

For further details of the calculations in this proof, see~\cite[Section~2]{BevanLambdaCalcs}.
\end{proof}

For our second theorem, we need to add extra sets of indecomposables to our constructions.
As before, we start with $r\geqslant3$, $s\geqslant2$, and odd $k\geqslant5$.
A suitable collection, $\HHH$, of extra sets of indecomposables satisfies the following two conditions:
\begin{bullets}
  \item Each set in $\HHH$ is disjoint from $Q^{r,s}$ and also disjoint from each $R_n^{r,s}$ for odd $n\geqslant k$.
  \item For each set $H\in\HHH$, the union 
  $H\cup Q^{r,s}$ is a downward closed set of indecomposables.
\end{bullets}
Given these conditions,
we define $\Phi_{r,s,k,\HHH}$ to be the family of those permutation classes whose indecomposables are the union of $Q^{r,s}$
together with an element of $\HHH$
and an element of $\FFF_n^{r,s}$ for each odd $n\geqslant k$:
$$
\Phi_{r,s,k,\HHH} \;=\;
\Big\{
\sumclosed \big( Q^{r,s} \:\cup\: H \:\cup\: S_k \:\cup\: S_{k+2} \:\cup\: S_{k+4} \:\cup\: \ldots
\big) \::\:
H \in \HHH \:\text{~and~}\:
S_n\in \FFF_n^{r,s}
, \, n=k,k+2,\ldots
\Big\} .
$$

We define 
our extra sets of indecomposables
by specifying 
an upper set of \emph{maximal} indecomposables, $U$, and a lower set of \emph{required} indecomposables, $L$.
If $S$ is a set of indecomposables, let
${\downarrow}S$ denote the downset consisting of those indecomposables that are subpermutations of elements of $S$.
Then, given $r,s\geqslant2$ and suitable sets of indecomposables $U$ and $L$, we use
$U\,{\Downarrow}_{r,s}\+L$
to denote the collection of those downward closed subsets of ${\downarrow}U\setminus Q^{r,s}$
that include the set of required indecomposables~$L$.
For instance, we have
$\FFF^{r,s}_n=\{\omega_n^{r,s}\}\,{\Downarrow}_{r,s}\+\{\omega_n^{3,2}\}$.


\begin{example}\label{exH}
Let's consider as an example the family $\Phi_{5,3,5,\HHH}$ where $\HHH 
= \{\omega_5^{7,1}\}\,{\Downarrow}_{5,3}\+\{\psi_7\}$;
see $\pi_0$ and $\mu_1$ in Figure~\ref{figColABPerms} below.
The sets in $\HHH$ consist of indecomposables that are subpermutations of $\omega_5^{7,1}$ but are not in $Q^{5,3}$.
There are six such indecomposables: $\psi_7$, $\psi_8$, $\omega_4^{6,1}$, $\omega_4^{7,1}$, $\omega_5^{6,1}$, and $\omega_5^{7,1}$.
The collection $\HHH$ consists of the nine nonempty downward closed subsets of these six permutations. 
See Figure~\ref{figHHasseDiag}.
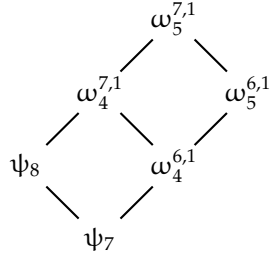
\begin{figure}[ht]
\newcommand{\hpt}[2]{\fill[white,radius=0.6] (#1) circle; \node at(#1){\small{#2}};}
$$
  \begin{tikzpicture}[scale=0.65]
    \draw [thick] (2,3)--(3.5,4.5)--(5,3)--(2,0)--(0.5,1.5)--(2,3)--(3.5,1.5);
    \hpt{2  ,0  }{$\psi_7$}
    \hpt{0.5,1.5}{$\psi_8$}            \hpt{3.5,1.5}{$\omega_4^{6,1}$}
    \hpt{2  ,3  }{$\omega_4^{7,1}$}    \hpt{5  ,3  }{$\omega_5^{6,1}$}
    \hpt{3.5,4.5}{$\omega_5^{7,1}$}
  \end{tikzpicture}
  \vspace{-4pt}
$$
  \caption{The Hasse diagram of the indecomposables that are elements of sets in $\HHH$}
  \label{figHHasseDiag}
\end{figure}

The set of indecomposables in a permutation class in our family consists of $Q^{5,3}$ together with an element of $\FFF_n^{5,3}$ for each odd $n\geqslant5$ and an extra set from $\HHH$.
By Lemma~\ref{lemmaQCount},
$Q^{5,3}$ contributes $(q_n)=(1,1,2,3,5,7,\overline8)$ to the enumeration of the indecomposables, and
for each odd $n\geqslant5$, there are ten distinct generalised digits
that enumerate
sets of indecomposables in $\FFF_n^{5,3}$, ranging between 1.1 and 1.2221.
Let $F_n$ consist of this set of generalised digits for odd $n\geqslant7$ and otherwise contain only 0.
The extra sets in $\HHH$
have seven distinct enumerations. These can be represented by the set of generalised digits
$$
H_1
\;=\;
\{0.0^61,\, 0.0^611,\, 0.0^6111,\, 0.0^612,\, 0.0^6121,\, 0.0^6122,\, 0.0^61221\}.
$$
Now, let $D_1=\{q_n+h:h\in H_1\}$ and for each $n>1$, let $D_n=\{q_n+f:f\in F_n\}$.

So, by construction, for every permutation class $\sumclosed\SSS\in\Phi_{5,3,5,\HHH}$ there is a corresponding sequence $(a_n)$, with each
$a_n\in D_n$, that enumerates $\SSS$.
The minimal enumeration sequence is $(\ell_n)\equiv(1, 1, 2, 3, 5, 7, 9, 10, \overline9)$ for which the growth rate is $\gr(\sumclosed(\ell_n)) \approx 2.36028$. Similarly,
the maximal enumeration sequence is $(u_n)\equiv(1, 1, 2, 3, 5, 7, 9, 11, 13, 14, 13, \overline{12})$ for which the growth rate is $\gr(\sumclosed(u_n)) \approx 2.36420$.

We now apply Lemma~\ref{lemmaGRInterval2}.
By Corollary~\ref{corGapIneqs}, the gap inequalities are
$$
\begin{array}{rcl}
         1                  - \gamma^{-2}                             &\leqslant& \gamma^{-1} + 2\+\gamma^{-2} + 2\+\gamma^{-3} + \gamma^{-4} , \qquad \\[4pt]
\gamma - 1 - 2\+\gamma^{-1} + \gamma^{-2} + \gamma^{-3}               &\leqslant& \gamma^{-1} + 2\+\gamma^{-2} + 2\+\gamma^{-3} + \gamma^{-4} , \\[4pt]
                              \gamma^{-2}               - \gamma^{-4} &\leqslant& \gamma^{-1} + 2\+\gamma^{-2} + 2\+\gamma^{-3} + \gamma^{-4} .
\end{array}
$$
These 
necessitate only that the growth rate not exceed $\gamma_{\max}\approx2.47098$.
Thus the growth rates of permutation classes in our example family $\Phi_{5,3,5,\HHH}$ form an interval.

For further details of the calculations in this example, see~\cite[Section~3]{BevanLambdaCalcs}.
\end{example}


The proof of our second theorem follows similar lines to this example.


\begin{repthm}{thmLambda}
Let $\lambda_B\approx2.35698$ be the unique positive root of $x^8-2\+x^7-x^5-x^4-2\+x^3-2\+x^2-x-1$.
Every value at least $\lambda_B$ is the growth rate of a permutation class.
\end{repthm}
\begin{proof}
In~\cite{Vatter2010b}, Vatter has shown that there are permutation classes of every growth rate at least $\lambda_A\approx2.48187$ (the unique real root of $x^5-2\+x^4-2\+x^2-2\+x-1$).
It thus suffices to exhibit permutation classes whose growth rates cover the interval $[\lambda_B,\lambda_A]$.
With the $\pi_i$ and $\mu_j$ as in Figures~\ref{figColABPerms}--\ref{figColEPerms} below,
we claim that the permutation classes in the following five families meet our needs:

\begin{tabular}{lccllccl}
$\quad$\textbf{Family} &$\!\!\!\!\!$\textbf{A} &$\!\!\!\!\!\!$\textbf{:}&
$\Phi_{5,3,7,\AAA}$ & $\!$where$\!$ &
$\AAA$ & $\!\!\!=\!\!\!$ & $\{\pi_1\}\,{\Downarrow}_{5,3}\+\{\mu_1\}$ \\[3pt]
$\quad$\textbf{Family} &$\!\!\!\!\!$\textbf{B} &$\!\!\!\!\!\!$\textbf{:}&
$\Phi_{5,3,5,\BBB}$ & $\!$where$\!$ &
$\BBB$ & $\!\!\!=\!\!\!$ & $\{\pi_2\}\,{\Downarrow}_{5,3}\+\varempty$ \\[3pt]
$\quad$\textbf{Family} &$\!\!\!\!\!$\textbf{C} &$\!\!\!\!\!\!$\textbf{:}&
$\Phi_{9,8,5,\CCC}$ & $\!$where$\!$ &
$\CCC$ & $\!\!\!=\!\!\!$ & $\{\pi_3\}\,{\Downarrow}_{9,8}\+\{\mu_2\}$ \\[3pt]
$\quad$\textbf{Family} &$\!\!\!\!\!$\textbf{D} &$\!\!\!\!\!\!$\textbf{:}&
$\Phi_{5,3,5,\DDD}$ & $\!$where$\!$ &
$\DDD$ & $\!\!\!=\!\!\!$ & $\{\pi_4,\pi_5\}\,{\Downarrow}_{5,3}\+\{\mu_3\}$ \\[3pt]
$\quad$\textbf{Family} &$\!\!\!\!\!$\textbf{E} &$\!\!\!\!\!\!$\textbf{:}&
$\Phi_{5,5,5,\EEE}$ & $\!$where$\!$ &
$\EEE$ & $\!\!\!=\!\!\!$ & $\{\pi_6,\pi_7,\pi_8\}\,{\Downarrow}_{5,5}\+\{\pi_6,\mu_3\} \;\cup\; \{\pi_6,\pi_7,\pi_8\}\,{\Downarrow}_{5,5}\+\{\mu_2,\mu_4,\mu_5\}$
\end{tabular}

\vspace{6pt}

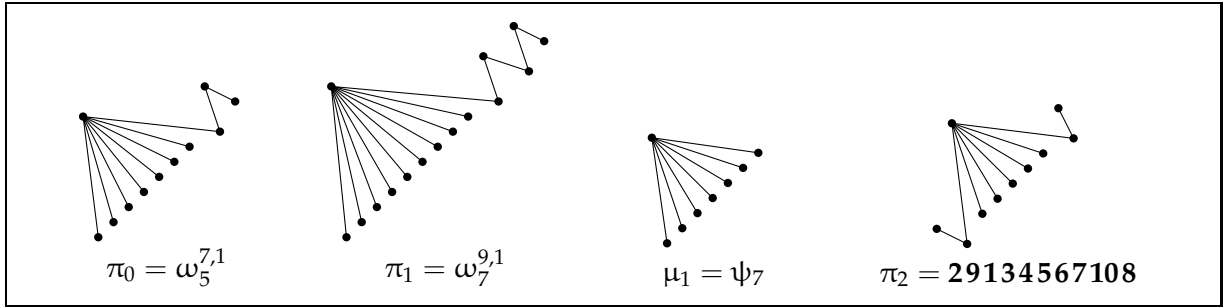
\begin{figure}[ht]
  \mybox{
  \vspace{3pt}
  $$
  \begin{tikzpicture}[scale=0.2]
    \plotpermnobox{11}{9, 1, 2, 3, 4, 5, 6, 7, 11, 8, 10}
    \draw [thin] (2,1)--(1,9)--(3,2);
    \draw [thin] (4,3)--(1,9)--(5,4);
    \draw [thin] (6,5)--(1,9)--(7,6);
    \draw [thin] (8,7)--(1,9)--(10,8);
    \draw [thin] (10,8)--(9,11)--(11,10);
    \node at(6,-1){$\;\;\pi_0=\omega_5^{7,1}$};
  \end{tikzpicture}
  \qquad\quad
  \begin{tikzpicture}[scale=0.2]
    \plotpermnobox{15}{11, 1, 2, 3, 4, 5, 6, 7, 8, 9, 13, 10, 15, 12, 14}
    \draw [thin] (2,1)--(1,11)--(3,2);
    \draw [thin] (4,3)--(1,11)--(5,4);
    \draw [thin] (6,5)--(1,11)--(7,6);
    \draw [thin] (8,7)--(1,11)--(9,8);
    \draw [thin] (10,9)--(1,11)--(12,10)--(11,13)--(14,12)--(13,15)--(15,14);
    \node at(8,-1){$\;\;\pi_1=\omega_7^{9,1}$};
  \end{tikzpicture}
  \qquad\quad
  \begin{tikzpicture}[scale=0.2]
    \plotpermnobox{8}{8, 1, 2, 3, 4, 5, 6, 7}
    \draw [thin] (2,1)--(1,8)--(3,2);
    \draw [thin] (4,3)--(1,8)--(5,4);
    \draw [thin] (6,5)--(1,8)--(7,6);
    \draw [thin] (1,8)--(8,7);
    \node at(4.5,-1){$\;\;\mu_1=\psi_7$};
  \end{tikzpicture}
  \qquad\quad
  \begin{tikzpicture}[scale=0.2]
    \plotpermnobox{10}{2,9,1,3,4,5,6,7,10,8}
    \draw [thin] (4,3)--(2,9)--(5,4);
    \draw [thin] (6,5)--(2,9)--(7,6);
    \draw [thin] (8,7)--(2,9);
    \draw [thin] (1,2)--(3,1)--(2,9)--(10,8)--(9,10);
    \node at(5.5,-1){$\;\pi_2=\mathbf{2\;\!9\;\!1\;\!3\;\!4\;\!5\;\!6\;\!7\;\!10\;\!8}$};
  \end{tikzpicture}
  \vspace{-6pt}
  $$
  }
  \caption{Permutations used to define Families~A and~B}\label{figColABPerms}
\end{figure}
We briefly outline the calculations concerning each of these families.
For the full details, see Section~4 of~\cite{BevanLambdaCalcs}.

\vspace{6pt}
\Needspace*{4\baselineskip}
\textbf{Family A:} $\,\Phi_{5,3,7,\AAA}$
\begin{bullets}
\raggedright
\item By Lemma~\ref{lemmaQCount}, $Q^{5,3}$ is enumerated by $(1,1,2,3,5,7,\overline8)$.
\item By Lemma~\ref{lemmaSizeOfF}, for each odd $n\geqslant7$, there are 10 distinct generalised digits that enumerate sets of indecomposables in $\FFF_n^{5,3}$, ranging between 1.1 and 1.2221.
\item There are 47 distinct enumerations of sets of indecomposables in $\AAA=\{\pi_1\}\,{\Downarrow}_{5,3}\+\{\mu_1\}$, ranging between $(0^7,1)$ and $(0^7,1,2,3,4,4,3,2,1)$.
\item The indecomposables in the smallest permutation class in {Family A} are enumerated by $(\ell_n) \equiv (1,1,2,3,5,7,8,\overline9)$.
\item The indecomposables in the largest permutation class in {Family A} are enumerated by $(u_n) \equiv (1,1,2,3,5,7,8,9,11,13,15,16,15,14,13,\overline{12})$.
\item By Lemma~\ref{lemmaGRSumClosed}, $\gr(\sumclosed(\ell_n)) = \lambda_B \approx 2.356983$; $\gr(\sumclosed(u_n)) \approx 2.359320$.
\item The gap inequalities require the growth rate not to exceed $\gamma_{\max}\approx2.470979$.
\end{bullets}
\vspace{6pt}
\Needspace*{4\baselineskip}
\textbf{Family B:} $\,\Phi_{5,3,5,\BBB}$
\begin{bullets}
\raggedright
\item By Lemma~\ref{lemmaQCount}, $Q^{5,3}$ is enumerated by $(1,1,2,3,5,7,\overline8)$.
\item By Lemma~\ref{lemmaSizeOfF}, for each odd $n\geqslant5$, there are 10 distinct generalised digits that enumerate sets of indecomposables in $\FFF_n^{5,3}$, ranging between 1.1 and 1.2221.
\item There are 29 distinct enumerations of sets of indecomposables in $\BBB=\{\pi_2\}\,{\Downarrow}_{5,3}\+\varempty$, ranging between $(0)$ and
$(0^5, 1, 2, 3, 3, 1)$.
\item The indecomposables in the smallest permutation class in {Family B} are enumerated by $(\ell_n) \equiv (1, 1, 2, 3, 5, 7,\overline9)$.
\item The indecomposables in the largest permutation class in {Family B} are enumerated by $(u_n) \equiv (1, 1, 2, 3, 5, 8, 11, 13, 14, 13,\overline{12})$.
\item By Lemma~\ref{lemmaGRSumClosed}, $\gr(\sumclosed(\ell_n)) \approx 2.359304$; $\gr(\sumclosed(u_n)) \approx 2.375872$.
\item The gap inequalities require the growth rate not to exceed $\gamma_{\max}\approx2.470979$.
\end{bullets}
\vspace{6pt}
\begin{figure}[ht]
  \mybox{
  \vspace{-6pt}
  $$
  \begin{tikzpicture}[scale=0.2]
    \plotpermnobox{10}{3,1,8,2,4,5,6,10,7,9}
    \draw [thin] (5,4)--(3,8)--(7,6);
    \draw [thin] (6,5)--(3,8);
    \draw [thin] (2,1)--(1,3)--(4,2)--(3,8)--(9,7)--(8,10)--(10,9);
    \node at(5.5,-1){$\;\pi_3=\mathbf{3\;\!1\;\!8\;\!2\;\!4\;\!5\;\!6\;\!10\;\!7\;\!9}$};
  \end{tikzpicture}
  \quad
  \begin{tikzpicture}[scale=0.2]
    \plotpermnobox{6}{2,5,1,3,6,4}
    \draw [thin] (1,2)--(3,1)--(2,5)--(6,4)--(5,6);
    \draw [thin] (2,5)--(4,3);
    \node at(3.5,-1){$\;\mu_2=\mathbf{251364}$};
  \end{tikzpicture}
  \quad
  \begin{tikzpicture}[scale=0.2]
    \plotpermnobox{6}{3, 1, 4, 5, 6, 2}
    \draw [thin] (2,1)--(1,3)--(6,2)--(5,6);
    \draw [thin] (4,5)--(6,2)--(3,4);
    \node at(3.5,-1){$\;\pi_4=\mathbf{314562}$};
  \end{tikzpicture}
  \quad
  \begin{tikzpicture}[scale=0.2]
    \plotpermnobox{9}{2, 8, 1, 3, 4, 5, 6, 9, 7}
    \draw [thin] (4,3)--(2,8)--(5,4);
    \draw [thin] (6,5)--(2,8)--(7,6);
    \draw [thin] (1,2)--(3,1)--(2,8)--(9,7)--(8,9);
    \node at(5,-1){$\;\pi_5=\mathbf{281345697}$};
  \end{tikzpicture}
  \quad
  \begin{tikzpicture}[scale=0.2]
    \plotpermnobox{4}{2,3,4,1}
    \draw [thin] (1,2)--(4,1)--(3,4);
    \draw [thin] (2,3)--(4,1);
    \node at(2.5,-1){$\;\mu_3=\mathbf{2341}$};
  \end{tikzpicture}
  \vspace{-9pt}
  $$
  }
  \caption{Permutations used to define Families~C,~D and~E}\label{figColCDEPerms}
\end{figure}
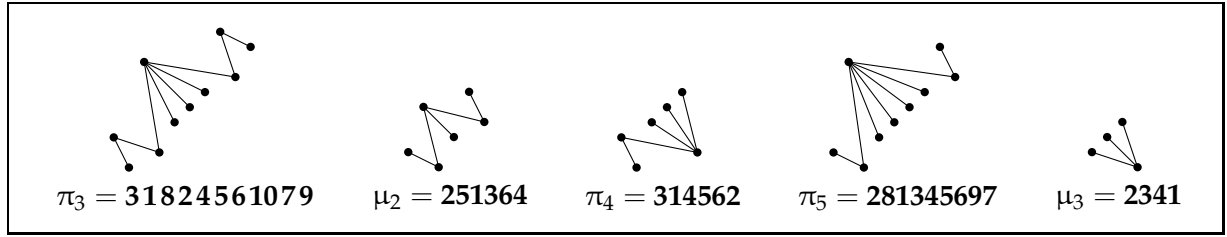
\Needspace*{4\baselineskip}
\textbf{Family C:} $\,\Phi_{9,8,5,\CCC}$
\begin{bullets}
\raggedright
\item By Lemma~\ref{lemmaQCount}, $Q^{9,8}$ is enumerated by $(1, 1, 2, 3, 5, 7, 9, 11, 13, 15,\overline{17})$.
\item By Lemma~\ref{lemmaSizeOfF}, for each odd $n\geqslant5$, there are 574 distinct generalised digits that enumerate sets of indecomposables in $\FFF_n^{9,8}$, ranging between 1.1 and 1.2345677654321.
\item There are 19 distinct enumerations of sets of indecomposables in $\CCC=\{\pi_3\}\,{\Downarrow}_{9,8}\+\{\mu_2\}$, ranging between $(0^5, 1)$ and
$(0^5, 1, 3, 4, 3, 1)$.
\item The indecomposables in the smallest permutation class in {Family C} are enumerated by $(\ell_n) \equiv (1, 1, 2, 3, 5, 8, 10, 12, 14, 16,\overline{18})$.
\item The indecomposables in the largest permutation class in {Family C} are enumerated by $(u_n) \equiv (1, 1, 2, 3, 5, 8, 13, 17, 20, 22, 26, 29, 33, 36, 39, 41, 43, 44,\overline{45})$.
\Needspace*{2\baselineskip}
\item By Lemma~\ref{lemmaGRSumClosed}, $\gr(\sumclosed(\ell_n)) \approx 2.373983$; $\gr(\sumclosed(u_n)) \approx 2.389043$.
\item The gap inequalities require the growth rate not to exceed $\gamma_{\max}\approx2.786389$.
\end{bullets}
\vspace{6pt}
\Needspace*{4\baselineskip}
\textbf{Family D:} $\,\Phi_{5,3,5,\DDD}$
\begin{bullets}
\raggedright
\item By Lemma~\ref{lemmaQCount}, $Q^{5,3}$ is enumerated by $(1,1,2,3,5,7,\overline8)$.
\item By Lemma~\ref{lemmaSizeOfF}, for each odd $n\geqslant5$, there are 10 distinct generalised digits that enumerate sets of indecomposables in $\FFF_n^{5,3}$, ranging between 1.1 and 1.2221.
\item There are 37 distinct enumerations of sets of indecomposables in $\DDD=\{\pi_4,\pi_5\}\,{\Downarrow}_{5,3}\+\{\mu_3\}$, ranging between $(0^3, 1)$ and
$(0^3, 1, 2, 2, 2, 2, 1)$.
\item The indecomposables in the smallest permutation class in {Family D} are enumerated by $(\ell_n) \equiv (1, 1, 2, 4, 5, 7,\overline9)$.
\item The indecomposables in the largest permutation class in {Family D} are enumerated by $(u_n) \equiv (1, 1, 2, 4, 7, 9, 11,\overline{12})$.
\item By Lemma~\ref{lemmaGRSumClosed}, $\gr(\sumclosed(\ell_n)) \approx 2.389038$; $\gr(\sumclosed(u_n)) \approx  2.430059$.
\item The gap inequalities require the growth rate not to exceed $\gamma_{\max}\approx 2.470979$.
\end{bullets}
\vspace{6pt}
\Needspace*{4\baselineskip}
\textbf{Family E:} $\,\Phi_{5,5,5,\EEE}$
\begin{figure}[t] 
  \mybox{
  \vspace{-6pt}
  $$
  \begin{tikzpicture}[scale=0.2]
    \plotpermnobox{4}{3,4,1,2}
    \draw [thin] (1,3)--(3,1)--(2,4)--(4,2)--(1,3);
    \node at(2.5,-1){$\;\pi_6=\mathbf{3412}$};
  \end{tikzpicture}
  \qquad
  \begin{tikzpicture}[scale=0.2]
    \plotpermnobox{7}{2, 6, 1, 3, 4, 7, 5}
    \draw [thin] (4,3)--(2,6)--(5,4);
    \draw [thin] (1,2)--(3,1)--(2,6)--(7,5)--(6,7);
    \node at(4,-1){$\;\pi_7=\mathbf{2613475}$};
  \end{tikzpicture}
  \qquad
  \begin{tikzpicture}[scale=0.2]
    \plotpermnobox{8}{3, 1, 4, 5, 6, 8, 2, 7}
    \draw [thin] (3,4)--(7,2)--(4,5);
    \draw [thin] (7,2)--(5,6);
    \draw [thin] (2,1)--(1,3)--(7,2)--(6,8)--(8,7);
    \node at(4.5,-1){$\;\pi_8=\mathbf{31456827}$};
  \end{tikzpicture}
  \qquad
  \begin{tikzpicture}[scale=0.2]
    \plotpermnobox{5}{2, 3, 4, 5, 1}
    \draw [thin] (1,2)--(5,1)--(3,4);
    \draw [thin] (2,3)--(5,1)--(4,5);
    \node at(3,-1){$\;\mu_4=\mathbf{23451}$};
  \end{tikzpicture}
  \qquad
  \begin{tikzpicture}[scale=0.2]
    \plotpermnobox{5}{2, 3, 5, 1, 4}
    \draw [thin] (1,2)--(4,1)--(3,5)--(5,4);
    \draw [thin] (2,3)--(4,1);
    \node at(3,-1){$\;\mu_5=\mathbf{23514}$};
  \end{tikzpicture}
  \vspace{-9pt}
  $$
  }
  \caption{Permutations used to define Family~E}\label{figColEPerms}
\end{figure}
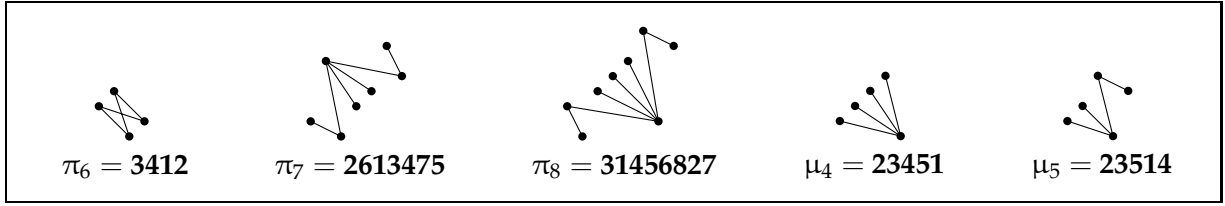
\begin{bullets}
\raggedright
\item By Lemma~\ref{lemmaQCount}, $Q^{5,5}$ is enumerated by $(1, 1, 2, 3, 5, 7, 9,\overline{10})$.
\item By Lemma~\ref{lemmaSizeOfF}, for each odd $n\geqslant5$, there are 26 distinct generalised digits that enumerate sets of indecomposables in $\FFF_n^{5,5}$, ranging between 1.1 and 1.234321.
\item There are 61 distinct enumerations of sets of indecomposables in $\EEE=\{\pi_6,\pi_7,\pi_8\}\,{\Downarrow}_{5,5}\+\{\pi_6,\mu_3\}$ $\,\cup\,$ $\{\pi_6,\pi_7,\pi_8\}\,{\Downarrow}_{5,5}\+\{\mu_2,\mu_4,\mu_5\}$, ranging between $(0^3, 1, 2, 1)$ and
$(0^3, 2, 3, 5, 4, 1)$.
\item The indecomposables in the smallest permutation class in {Family E} are enumerated by $(\ell_n) \equiv (1, 1, 2, 4, 7, 8, 10,\overline{11})$.
\item The indecomposables in the largest permutation class in {Family E} are enumerated by $(u_n) \equiv (1, 1, 2, 5, 8, 12, 14, 13, 14, 16, 17,\overline{18})$.
\item By Lemma~\ref{lemmaGRSumClosed}, $\gr(\sumclosed(\ell_n)) \approx 2.422247$; $\gr(\sumclosed(u_n)) \approx  2.485938 > \lambda_A$.
\item The gap inequalities require the growth rate to be at least $\gamma_{\min}\approx 2.363728$, but not to exceed $\gamma_{\max}\approx 2.489043$.
\end{bullets}
\vspace{6pt}
\Needspace*{13\baselineskip}
Here is a summary:
\begin{center}
\renewcommand{\arraystretch}{1.25}
\begin{tabular}{|c|l|l|}
  \hline
  & 
    \emph{Enumeration of smallest and largest sets of indecomposables}
  & \emph{Interval covered} \\
  \hline
  {~A~} &
    $\hspace{-\arraycolsep}\begin{array}{l}
      (1,1,2,3,5,7,8,\overline9) \\
      (1,1,2,3,5,7,8,9,11,13,15,16,15,14,13,\overline{12})
    \end{array}\hspace{-\arraycolsep}$ &
    $2.356983$ -- $2.359320$ \\
  \hline
  {B} &
    $\hspace{-\arraycolsep}\begin{array}{l}
      (1,1,2,3,5,7,\overline9) \\
      (1,1,2,3,5,8,11,13,14,13,\overline{12})
    \end{array}\hspace{-\arraycolsep}$ &
    $2.359304$ -- $2.375872$ \\
  \hline
  {C} &
    $\hspace{-\arraycolsep}\begin{array}{l}
      (1,1,2,3,5,8,10,12,14,16,\overline{18}) \\
      (1,1,2,3,5,8,13,17,20,22,26,29,33,36,39,41,43,44,\overline{45})
    \end{array}\hspace{-\arraycolsep}$ &
    $2.373983$ -- $2.389043$ \\
  \hline
  {D} &
    $\hspace{-\arraycolsep}\begin{array}{l}
      (1,1,2,4,5,7,\overline9) \\
      (1,1,2,4,7,9,11,\overline{12})
    \end{array}\hspace{-\arraycolsep}$ &
    $2.389038$ -- $2.430059$ \\
  \hline
  {E} &
    $\hspace{-\arraycolsep}\begin{array}{l}
      (1,1,2,4,7,8,10,\overline{11}) \\
      (1,1,2,5,8,12,14,13,14,16,17,\overline{18})
    \end{array}\hspace{-\arraycolsep}$ &
    $2.422247$ -- $2.485938$ \\
  \hline
\end{tabular}
\end{center}
Thus we have five intervals of growth rates that cover $[\lambda_B,\lambda_A]$.
\end{proof}

\subsubsection*{Acknowledgements}
The author is grateful to Vince Vatter for suggesting that it might be worthwhile investigating whether his conjecture 
concerning the behaviour below $\lambda_A$ may in fact be false.
He would also like to thank Vince, Robert Brignall and two referees for reading earlier drafts of this paper; their feedback
resulted in significant improvements to its presentation.

\emph{S.D.G.}

\Needspace*{7\baselineskip}
\def\UrlBreaks{\do\/}
\bibliographystyle{plain}
{\footnotesize\bibliography{mybib}}

\HIDE
{
\appendix
\section{Computer-aided calculations}

\includepdf[pages=-]{intervalCalculationsPDF.pdf}

} 

\end{document}